\documentclass[pdflatex,sn-mathphys-num]{sn-jnl}
\usepackage{graphicx}%
\usepackage{multirow}%
\usepackage{amsmath,amssymb,amsfonts}%
\usepackage{amsthm}%
\usepackage{mathrsfs}%
\usepackage[title]{appendix}%
\usepackage{xcolor}%
\usepackage{bbm}
\usepackage{textcomp}%
\usepackage{manyfoot}%
\usepackage{booktabs}%
\usepackage{algorithm}%
\usepackage{algorithmicx}%
\usepackage{algpseudocode}%
\usepackage{listings}%
\let\orcidlogo\relax
\usepackage{orcidlink}
\theoremstyle{thmstyleone}%
\newtheorem{theorem}{Theorem}
\newtheorem{proposition}[theorem]{Proposition}%
\theoremstyle{thmstyletwo}%
\newtheorem{remark}{Remark}%
\theoremstyle{thmstylethree}%
\newtheorem{definition}{Definition}%
\newtheorem{lemma}{Lemma}

\begin{document}
\title[Quantitative Estimates for Mean-Field Limits]
{Quantitative Estimates for Mean-Field Limits and Correlation
Functions through a Duality Framework}

\author*[1]{
\fnm{Nadia} \sur{Khoury}\, \orcidlink{0009-0008-4439-6656}
}
\author[1]{\fnm{Pierre-Emmanuel} \sur{Jabin}\, \orcidlink{0000-0001-5998-0347}
}

\affil[1]{
  \orgdiv{Department of Mathematics},
  \orgname{The Pennsylvania State University},
  \orgaddress{
    \city{University Park},
    \postcode{16802},
    \state{Pennsylvania},
    \country{USA}
  }
}

\abstract{
We investigate the mean-field limit for interacting particle systems through a duality-based framework and obtain quantitative estimates on the convergence of marginals as well as on correlation functions. The analysis applies to second-order Vlasov systems with pairwise interactions under a mean-field scaling and relies on a hierarchy of dual cumulants associated with the particle dynamics. In particular, for merely square--integrable interaction forces, we derive the natural fluctuation--scale rate \(\mathcal{O}(N^{-1/2})\). By introducing an iterative argument on the hierarchy of dual cumulants, we leverage this bound to recover the optimal mean-field rate \(\mathcal{O}(N^{-1})\) and to obtain robust estimates on the dual cumulants, at the expense of corresponding regularity assumptions on the interaction kernel. Finally, using the relation between dual and direct correlations, we transfer these bounds to direct cumulants, yielding refined information on correlations and deviations from chaos. The approach provides a unified framework for simultaneously controlling the mean-field limit and the higher-order correlation structure of the particle system.

}


\maketitle

\begingroup
\renewcommand{\thefootnote}{}
\footnotetext{%
\noindent
E-mail addresses:
\href{mailto:npk5371@psu.edu}{npk5371@psu.edu}
(N. Khoury);
\href{mailto:pejabin@psu.edu}{pejabin@psu.edu}
(P.-E. Jabin).
}
\addtocounter{footnote}{-1}
\endgroup

\medskip

\tableofcontents

\section{Introduction}\label{sec1}

We investigate the collective behavior of a large system of \(N\) identical
particles interacting pairwise through forces under a mean-field scaling. We consider Newtonian dynamics with a possible velocity diffusion. Let \(\Omega = \mathbb{R}^d\) or \(\mathbb T^d\), and let $\mathcal{D} := \Omega \times \mathbb{R}^d$. For \(i=1,\dots,N\) let \((X_i(t),V_i(t))\in\mathcal D\) denote the position and velocity of the \(i\)-th particle at time \(t\). The particle system is governed by
\begin{equation}\label{eq:SDE_system}
\begin{cases}
dX_i = V_i\,dt, \\[2mm]
dV_i = \displaystyle \frac{1}{N-1}\sum_{\substack{j=1\\ j\neq i}}^N K(X_i-X_j)\,dt
+\sqrt{2\sigma}\,dW_i,
\end{cases}
\end{equation}
where \(K:\Omega\to\mathbb{R}^d\) denotes the pairwise interaction force kernel, \(\sigma\ge 0\) is the diffusion coefficient, and $W_i$ are $N$ independent standard Wiener Processes (Brownian motions) in $\mathbb{R}^d$. The stochastic term in \eqref{eq:SDE_system} should be understood in the Itô sense.

Although the formulation above allows for any \(\sigma\ge0\), our main interest is the deterministic case \(\sigma=0\), and the vanishing diffusion setting $\sigma=\sigma_N\to 0$ as $N\to \infty$. Our proofs are essentially transparent in terms of diffusion, so that we can allow any value of $\sigma\geq 0$. For simplicity, however, we take $\sigma=0$ in the rest of the article.

Our goal is to derive a precise statistical description of the particles' dynamics. In the deterministic case \(\sigma=0\), the particle system~\eqref{eq:SDE_system} leads to the well-known Liouville equation
\begin{equation}\label{eq:Liouville_equation}
\begin{aligned}
\partial_t F_{N} + \sum_{i=1}^N \left(v_i \cdot \nabla_{x_i} F_{N} + \frac{1}{N-1} \sum_{\substack{j=1\\ j\neq i}}^N K(x_i -x_j) \cdot \nabla_{v_i} F_{N}\right) = 0,
\end{aligned}
\end{equation}
\noindent
where \(F_{N}\) is the full joint law of the system and is correspondingly a probability density on the \(N\)-particle phase space \( \mathcal{D}^N = \left(\Omega \times \mathbb{R}^d\right)^N \). In addition, we assume that at initial time particles are \(f^0\)-chaotic in the sense of 
\begin{equation}\label{eq:chaotic}
F_N|_{t=0} = \left(f^0\right)^{\otimes N} ,
\end{equation}
\noindent
for some bounded probability density \(f^0 \in \mathcal{P} (\mathcal{D}) \cap L^{\infty}(\mathcal{D})\). This assumption could be relaxed to some extent by taking $F_{N}|_{t=0}$ close enough to $\left(f^0\right)^{\otimes N}$, in some strong topology, but we keep~\eqref{eq:chaotic} for simplicity here. 

The full joint law $F_N$ is however posed in a space that grows with $N$, and for this reason, the so-called marginal distributions offer an easier description of the  collective behavior of particles. For each $0\le m\le N$, the $m$--th particle marginal is defined by
\[
F_{N,m}(z_1,\dots,z_m)
:= \int_{\mathcal{D}^{\,N-m}} F_{N}(z_1,\dots,z_N)\,
dz_{m+1}\cdots dz_N,
\]
where we denote $z_i=(x_i,v_i)$ for simplicity. The marginal $F_{N,m}$ can also be defined as the law of the partial or reduced $(X_1,V_1),\ldots,(X_m,V_m)$. As such $F_{N,m}$ cannot satisfy a closed equation such as~\eqref{eq:Liouville_equation} for $F_N$. This is what makes obtaining any estimates on the $F_{N,m}$ challenging.

The dynamics of the particle system is expected to follow the so-called mean-field density~$f$ which solves the limiting Vlasov equation, 
\begin{equation}\label{eq:vlasov_eq}
\partial_t f + v \cdot \nabla_x f + (K * f) \cdot \nabla_v f = 0,\qquad f|_{t=0} = f^0,
\end{equation}
where we denote \(K * f (x) := \int_{\mathcal{D}} K (x-\bar{x}) f(\bar{x},\bar{v}) d \bar{x} d \bar{v} \). We further assume that \(f>0\) $a.e$ on \( [0,T]\times \mathcal{D}\), so that the logarithmic derivatives and the quotients involving \(f\) appearing below are well defined almost everywhere. Assuming tensorized initial data,
\[
F_{N,m}|_{t=0}=(f^0)^{\otimes m},
\]
one expects propagation of chaos, namely
\[
F_{N,m}-f^{\otimes m}\longrightarrow0,
\qquad\text{as }N\to\infty.\]

Our first main result consists in proving the optimal convergence rate on the marginals. We equip \(C_c(\mathcal D^m)\) with the \(L^\infty\)-norm and denote its dual by
\(C_c(\mathcal D^m)^\star\).
\begin{theorem}\label{thm1}

Let $F_N \in L^{\infty}_{loc}(\mathbb{R}^+; L^1(\mathcal D^N)\cap L^\infty(\mathcal D^N))$ be a global weak duality solution of the Liouville equation~\eqref{eq:Liouville_equation}, as in Appendix A, with $f^0$--chaotic initial data~\eqref{eq:chaotic}, where $f^0 \in \mathcal P(\mathcal D)\cap L^\infty(\mathcal D)$. Let $f \in L^{\infty}_{loc}(\mathbb{R}^+; \mathcal P(\mathcal D)\cap L^\infty(\mathcal D))$ be a bounded weak solution of the Vlasov equation~\eqref{eq:vlasov_eq} with initial data $f^0$ such that $\operatorname{ess}\sup_{x_\star} \int |\nabla_{v_\star}\log f(z_\star)|^2 f(z_\star)\,dv_\star < \infty$

\begin{enumerate}

\item Assume that the interaction kernel satisfies 
\[
K \in \big( L^4(\Omega;\mathbb R^d) + L^\infty(\Omega;\mathbb R^d) \big)\cap H^1(\Omega;\mathbb R^d),
\qquad
K*f \in L^\infty(0,T; W^{2,\infty}(\Omega ; \mathbb{R}^d)),
\]
and that the mean-field density $f$ satisfies
\[
\int_0^T \left( \int_{\mathcal D} \Big( |\nabla_v \log f|^2 + \Big|\tfrac{1}{f}\nabla^2_{zv} f\Big|^2 \Big) f \right)^{\frac12} < \infty
\]
Then there exists $\Delta t_1>0$ such that for all $0 \le t \le \Delta t_1$ and every fixed $1 \leq m \leq N$
\[
\|F_{N,m}(t) - f^{\otimes m}(t)\|_{C_c(\mathcal D^m)^\star}
\le C N^{-1/2},
\]
for some constant $C$ independent of $N$, $T$, and $t$.\\

\item Assume that the interaction kernel satisfies
\[
K \in W^{1,\infty}(\Omega;\mathbb R^d)\cap H^2(\Omega;\mathbb R^d),
\qquad
K*f \in L^\infty(0,T; W^{3,\infty}(\Omega ; \mathbb{R}^d)),
\]
and that the mean-field density $f$ satisfies
\[
\int_0^T \left( \int_{\mathcal D} \Big( |\nabla_v \log f|^2 
+ \Big|\tfrac{1}{f}\nabla^2_{zv} f\Big|^2 
+ \Big|\tfrac{1}{f}\nabla^3_{zzv} f\Big|^2 \Big) f \right)^{\frac12} < \infty.
\]
Then there exists $\Delta t_2>0$ (possibly smaller than $\Delta t_1$) such that for all $0 \le t \le \Delta t_2$ and every fixed $1 \leq m \leq N$
\[
\|F_{N,m}(t) - f^{\otimes m}(t)\|_{C_c(\mathcal D^m)^\star}
\le C N^{-1},
\]
for some constant $C$ independent of $N$, $T$, and $t$. \\

\end{enumerate}

\end{theorem}

\begin{remark}
The constant \(C\) in Theorem \ref{thm1} depends on the marginal order \(m\). We simply denote this constant by \(C\) to avoid cumbersome notation.
\end{remark}
Foundational works such as \cite{dobrushin1979vlasov} and \cite{braun1977vlasov} require strong regularity assumptions on the interaction kernel, typically \(K\in W^{1,\infty}\), with the former explicitly yielding the classical convergence of the empirical measure at a rate of \(\mathcal{O}(N^{-1/2})\). By contrast, the present approach could actually reach a similar rate under a weaker assumption \(K\in H^{1}\). This \(O(N^{-1/2})\) rate for the convergence of the empirical measure is known to be optimal in general; see in particular \cite{bernou2025uniform} and \cite{wang2023gaussian}.

A comparable \(O(N^{-1/2})\) rate was obtained in~\cite{jabin2016mean}, which only assumes \(K\in L^{\infty}\) through a relative entropy approach between the full law of the particles and the tensorized mean-field limit. This applies to second--order models such as considered here, with possibly vanishing diffusion or deterministic dynamics. In that sense, this result is obtained under a weaker regularity assumption on the kernel while yielding the same empirical-measure rate. However, the \(O(N^{-1/2})\) scale also seems to be the natural limitation of that approach.

The situation is different at the level of the marginals, where additional cancellations may lead to faster convergence. A first example is provided by \cite{duerinckx2021size}, which yields \(O(N^{-1})\)-type corrections for the one-particle marginal through sharp estimates on higher-order correlation functions. 
In a different direction, \cite{lacker2023hierarchies} also obtains sharp \(O(N^{-1})\)-type marginal estimates, but requires a full and non-vanishing diffusion, which cannot hold in the present setting where we typically consider  deterministic or vanishing-diffusion dynamics.

We focus in the present work on optimal rates of convergence, which are understood to require smooth interactions. But obviously, many applications instead involve some sort of singular interactions: the most classical example being Coulombian interactions where $K(x)=\gamma\,\frac{x}{|x|^d}$. Unfortunately we have very few results proving the mean-field limits for second-order kinetic models with singular interactions. In the absence of diffusion, \cite{hauray2007n}  and \cite{hauray2015particle} treat Vlasov systems with singular interaction forces of type \(|x|^{-\alpha}\), \(\alpha<1\), and also allow stronger singularities in the presence of a cut-off. Singular interactions with  \(N\)-dependent cut-off had been considered earlier in~\cite{ganguly1989convergence,ganguly1991simulation,victory1991convergence,wollman2000approximation}. More recent works obtained substantially sharper cutoff scales; see \cite{lazarovici2016vlasov,lazarovici2017mean, huang2020mean}. Extensions to bounded, possibly discontinuous interaction forces were established in~\cite{jabin2016mean}. For second-order systems with singular interactions, the Vlasov–Poisson–Fokker–Planck equation with a polynomial cut-off was studied in~\cite{carrillo2019propagation}.  In special models such as Cucker–Smale flocking, kinetic and measure-valued methods yield mean-field convergence for certain singular kernels, see for example~\cite{mucha2018cucker}.

In the case with a non-vanishing diffusion in velocity, \cite{huang2020mean} extends the above cut-off results to stochastic dynamics. \cite{bresch2025new} provides the first derivation of the full Vlasov–Poisson–Fokker–Planck equation in higher dimensions with no truncation in the repulsive case.  However the presence of diffusion plays a crucial structural role again there, so the method cannot apply to the present deterministic or vanishing-diffusion regime.

We also mention~\cite{serfaty2020mean}, which is based on the modulated energy method, where the convergence to the deterministic Vlasov--Poisson is obtained but only in the so-called monokinetic regime.

Finally, \cite{bresch2024duality} is particularly close in terms of the approach to the present result. That work derived the mean-field limit under assumptions as weak as \(K\in L^{2}\), by means of a duality method and an analysis of dual correlations. A quantitative rate was also obtained if $K$ is smoother, namely $K\in H^s$ for some $s>0$. However, the quantitative rate remains suboptimal, and does not in general reach $N^{-1/2}$ if $K$ is not bounded. In this respect, the novelty of the present work is to take advantage of  additional structure in order to obtain sharper quantitative estimates, while keeping the regularity assumptions as low as possible.

The marginals are not expected to converge at a faster rate than $O(N^{-1})$. A convenient way to quantify further deviations from the mean-field limit is through the analysis of higher-order correlation structures, referred to in the present paper as direct cumulants. In the classical setting, these are defined as combinations of the marginals through a cluster expansion over partitions. 

There exist various notions of direct cumulants. One of the most classical ones is probably the following,  denoting by $F_N$ the joint density, 
\[
F_{N}(z_1,...,z_N)
= \sum_{\pi \in \mathcal{P}(\{1,\dots,N\})} \;\prod_{\sigma \in \pi} G_{N,|\sigma|}(z_\sigma),
\]
where the sum runs over all partitions of $\{1,\dots,N\}$, and the product is over blocks of the partition $\sigma \in \pi$, and $z_\sigma := (z_i)_{i \in \sigma}$. To guarantee uniqueness of this expansion, we set $G_{N,1}(z_1) := F_{N,1}(z_1)$ as the first marginal, and impose the following cancellation property
\[\int_{\mathcal D} G_{N,m}(z_1,\dots,z_m)\,dz_j = 0 , \quad \forall 1 \leq j \leq m \,\,\, \ \ \ m \geq2.\]
In addition, $\{G_{N,m}\}_{1 \leq m \leq N}$ can also be defined by Möbius inversion as
\[
G_{N,m}
= \sum_{\pi \in \mathcal{P}(\{1,\dots,m\})}
(|\pi|-1)!(-1)^{|\pi|-1} \prod_{\sigma \in \pi} F_N^{(|\sigma|)}.
\]
The $G_{N,m}$ correspond to some well-known physical quantities: $G_{N,1}$ is just the 1-particle distribution for example while $G_{N,2}$ measures the correlations in the system.

However, we are specifically interested here in the deviation from the mean-field limit, for which it is more convenient to express the $N$-particle distribution as a perturbation of the factorized mean--field density $f^{\otimes N}$. 
The standard marginals contain highly redundant information: for instance, if the first marginal $F_{N,1}$ is already close to $f$, then the primary object of interest is how far the higher-order marginals deviate from their corresponding tensorized profiles, such as $F_{N,2} - f^{\otimes 2}$. 
To systematically isolate these fluctuations without the combinatorial complexity of classical partition-based cumulants, we organize the joint density through a direct linear expansion.
Concretely, if $F_N$ denotes the joint density, one can consider the cluster expansion
\[
F_N(z_1,\cdots,z_N)
= f^{\otimes N}\sum_{n=0}^N \sum_{\sigma \in \mathcal{P}_n^N} \kappa_{N,n}(z_\sigma),
\]
with the convention $\kappa_{N,0} = 1$ where $\mathcal{P}_n^N$ denotes the set of all subsets of $\{1,\dots,N\}$ with $n$ elements and $z_\sigma = (z_i)_{i\in \sigma}\in \mathcal{D}^n$.
\noindent
In order for the functions $\kappa_{N,n}$ to be  uniquely defined, we impose the cancellation property
\[
\int_{\mathcal{D}} \kappa_{N,n}(z_1,\cdots,z_n)\, f(z_i)\,dz_i = 0 \quad \text{for all } 1 \le i \le n.
\]

From a physical and probabilistic perspective, both of these definitions are designed to capture correlations beyond factorization and provide a natural framework for mean--field analysis. Indeed, while the marginals $F_{N,n}$ contain both independent and correlated contributions, the direct cumulants $\kappa_{N,n}$ remove lower-order factorizations and provide a structured way to quantify correlations at each order. In particular, propagation of chaos can be characterized by the decay of $\kappa_{N,n}$ for all $n \geq 2$ as $N \to \infty$.

Throughout the paper,weighted spaces over \(\mathcal{D}^n\) are understood with respect to the product measure \(f^{\otimes n}dz_{[n]}\); in particular,
\[
L_f^2(\mathcal{D}^n):=L^2\!\left(\mathcal{D}^n,f^{\otimes n}\right),
\]
and
\[
L_f^1(\mathcal{D}^n):=L^1\!\left(\mathcal{D}^n,f^{\otimes n}\right).
\]
We now state our main result, which provides quantitative bounds on the direct cumulants $\kappa_{N,n}$ introduced above.
\begin{theorem}\label{direct_cumulants_bound}
Let $k\geq 1$ be a regularity index, and let $F_N\in L^\infty_{\mathrm{loc}}
\bigl(\mathbb{R}_+;L^1(\mathcal D^N)\cap L^\infty(\mathcal D^N)\bigr)$ be a global weak-duality solution of the Liouville equation~\eqref{eq:Liouville_equation}, in the sense of Appendix~\ref{secA1}, with initial data \(F_N^0\). Let the interaction kernel satisfy \( K \in \mathscr{R}_k(\Omega; \mathbb{R}^d)
 \cap H^{k}(\Omega; \mathbb{R}^d) \) where 
\[
\mathscr{R}_k(\Omega; \mathbb{R}^d)
 :=
\begin{cases}
L^{\infty}(\Omega; \mathbb{R}^d) + L^4(\Omega; \mathbb{R}^d) & \text{if } k = 1, \\
W^{k-1,\infty}(\Omega; \mathbb{R}^d) & \text{if } k \geq 2.
\end{cases}
\] 
Assume further that $f \in L^{\infty}_{loc}(\mathbb{R}^+; \mathcal P(\mathcal D)\cap L^\infty(\mathcal D))$ is a bounded weak solution of the Vlasov equation~\eqref{eq:vlasov_eq} with initial data \(f^0\), and that \(f>0\) almost everywhere. Assume that \(f\) satisfies $\operatorname{ess}\sup_{x_\star} \int |\nabla_{v_\star}\log f(z_\star)|^2 f(z_\star)\,dv_\star < \infty$, along with the regularity conditions
\[
\int_{0}^{T} \left( \int_{\mathcal{D}} \Big( |\nabla_{v} \log f|^{2} + \sum_{j=1}^k \Big| \tfrac{1}{f} \nabla^{j}_{z} \nabla_{v} f \Big|^{2} \Big) f \right)^{\frac{1}{2}} < \infty,
\qquad
K * f \in L^{\infty}\left(0, T; W^{k+1,\infty}(\Omega; \mathbb{R}^d)\right).
\]
Let \(\{\kappa_{N,n}\}_{0\leq n\leq N}\) denote the direct cumulants
of \(F_N\) relative to the mean-field density \(f\), as defined above. For \(0\le n\le N\), let
\begin{equation} \label{Lambda_{N,n,k}}
\Lambda_{N,n,k}
:= \max\!\left\{
\|\bar{\kappa}_{N,n}^{0}\|_{L^2_{f^0}},\,
 \sum_{j=1}^{k}
\left\| \frac{\nabla^j(f^0\bar{\kappa}_{N,n}^{0})}{f^0} \right\|_{L^2_{f^0}}
\right\},
\end{equation}
where $\bar{\kappa}_{N,n} := \binom{N}{n}^{1/2} \kappa_{N,n}$ are the rescaled direct cumulants, $\bar{\kappa}_{N,n}^0 = \bar{\kappa}_{N,n}|_{t=0}$, and $f^0 = f|_{t=0}$. Assume further that the initial data satisfies
$\sum_{n=0}^{N} C^{nk}\Lambda_{N,n,k} \lesssim 1$. Then, there exists $\Delta t_k > 0$ such that for all $0 \leq t \leq \Delta t_k$ the $k$-th order direct cumulant satisfies
\begin{equation}\label{eq:direct_L1_bound1}
\|\kappa_{N,k}(t) \|_{L^1_f}
\lesssim C_k\, N^{-k/2}.
\end{equation}
\end{theorem}

\begin{remark}
Unlike the previous result, Theorem \ref{direct_cumulants_bound} does not require tensorized initial data. Instead, we assume only the initial weighted cumulant bound above.
\end{remark}

There exist only a few results providing higher-order estimates on correlation or cumulant structures. For classical combinatorial cumulants, denoted here as $G_{N,k}$, the optimal scaling of the \(k\)-particle correlation, namely \(G_{N,k}=O(N^{1-k})\), has been obtained in \cite{duerinckx2021size} and \cite{hess2025higher}. \cite{hess2025higher} inherently relies on a non-degenerate, non-vanishing diffusion, but only requires bounded interaction kernels.  
To the best of our knowledge, the only article prior to the present work capable of handling entirely deterministic dynamics ($\sigma=0$) is~\cite{duerinckx2021size}. \cite{duerinckx2021size} demands a strong regularity, namely $K\in W^{k+1,\infty}$ to obtain \(G_{N,k}=O(N^{1-k})\) in some weak Sobolev norm. 

Our estimate for the direct cumulant $\kappa_{N,k}$ only requires the weaker assumption \(K\in W^{k-1,\infty}\cap H^k\), at the expense of yielding the slower rate \(\kappa_{N,k} = \mathcal{O}(N^{-k/2})\). 
However, this slower rate is the critical one to understand fluctuations, through a Central Limit Theorem (CLT) for example. 
Moreover, we only require this same $\mathcal{O}(N^{-k/2})$ scaling to hold at the initial time $t=0$, whereas~\cite{duerinckx2021size} considers fully tensorized initial data, which could likely be relaxed but not to \( \kappa_{N,k} = O(N^{-k/2})\).

We note that the rate $\kappa_{N,k} = \mathcal{O}(N^{-k/2})$ appears to be a fundamental feature of our method, because we propagate estimates directly on the so-called dual cumulants, the structure of the dual hierarchy naturally hits a wall at this fluctuation scale for the $\kappa_{N,k}$.
This scaling directly mirrors the rate obtained in the recent work~\cite{duerinckx2025correlation}, which requires non-degenerate diffusion and only $K\in L^2$. 
Because the diffusion is non-vanishing in~\cite{duerinckx2025correlation}, their analytical approach is entirely different; however, their framework starts by establishing rates on the $\kappa_{N,k}$ before leveraging them to derive the classical rates on $G_{N,k}$. 
This suggests that the optimal rate of convergence for the direct cumulants $\kappa_{N,k}$ is inherently limited to $\mathcal{O}(N^{-k/2})$, instead of the faster rate on the $G_{N,k}$.

Cumulant methods also play an important role in other settings. In the derivation of the Lenard--Balescu equation, \cite{duerinckx2021lenard} used a cumulant hierarchy to derive the first nontrivial correction and to identify it with the Lenard--Balescu operator. In the Boltzmann setting, related cumulant and cluster-expansion ideas have proved critical in the derivation of the Boltzmann equation; see, for instance, the long-time equilibrium result \cite{bodineau2024long} and the recent long-time derivation from hard-sphere dynamics in \cite{deng2024long}.

 We emphasize that the mean-field limit is much better understood for so-called first-order systems, with several recent breakthroughs.  \cite{serfaty2020mean} developed a modulated-energy method to obtain quantitative mean-field convergence for Coulomb and more general deterministic Riesz-type flows, exploiting weak--strong stability estimates to control the distance between the empirical measure and the limiting density. For the 2D Navier--Stokes system, early propagation-of-chaos results go back to~\cite{osada1986propagation} and were later strengthened in~\cite{fournier2014propagation}. \cite{jabin2018quantitative} introduced a relative entropy method yielding the quantitative propagation of chaos for the 2D Navier-Stokes. Uniform-in-time propagation of chaos for the 2D vortex model and related singular stochastic systems was subsequently obtained in~\cite{guillin2024uniform}. The relative entropy method was later combined with the modulated-energy method in~\cite{bresch2023mean}, providing a framework adapted to singular attractive kernels with diffusion and yielding quantitative convergence to Keller--Segel. \cite{nguyen2022mean} extended the modulated-energy method to general Riesz flows with possible multiplicative transport noise, while \cite{rosenzweig2023global} established global-in-time quantitative convergence for singular Riesz-type diffusive dynamics while the so-called $\log$ gas is thoroughly treated in any dimension in~\cite{chodron2025attractive}. We also mention the classical mean-field limits toward 2D Euler and related flows that were established in~\cite{goodman1990convergence,goodman1991new,schochet1995weak,schochet1996point}. 

 As already mentioned, our work is closest to~\cite{bresch2024duality}. Instead of directly studying the propagation of the tensorized structure of the $N$-particle density $F_N$ solving the Liouville equation, that work reformulates the problem through a dual Liouville equation evolving backward in time. The dual equation is initialized with observables having a linear structure, and the mean-field limit problem is then reformulated as the study of how this linear structure propagates under the dual structure.

The advantage of this reformulation is that deviations from linearity can be naturally measured through linear correlation functions associated with the dual equation. These quantities behave more favorably than the usual nonlinear correlations appearing in classical propagation-of-chaos formulations. In particular, they allow the use of Hilbertian techniques and energy-type estimates to obtain a priori bounds. The resulting quantities satisfy a BBGKY-type hierarchy of equations, and the analysis of this hierarchy becomes tractable because the potentially singular derivative losses only appear in perturbative terms that vanish in the mean-field limit.

Compared with~\cite{bresch2024duality}, the present work differs in three main respects. First, while~\cite{bresch2024duality} establishes mean-field convergence at rate \(N^{-1/2}\) under \(H^1\cap L^\infty\)-type assumptions on the interaction kernel, we only require \(K\in H^1\cap L^4\) at that level, and under the stronger assumption \(K\in H^2\cap W^{1,\infty}\) our approach recovers the optimal rate \(N^{-1}\) for the convergence of the marginals to the mean-field limit. Second, under the more general assumption \(K\in H^k\cap W^{k-1,\infty}\), we show that the duality estimates can be iterated beyond the first levels: the remainder term arising at order \(k\) is analyzed by using the decomposition already obtained at order \(k-1\). In this way, each new step of the argument gains an additional factor \(N^{-1/2}\), at the expense of one extra derivative assumption on \(K\), leading to bounds of order \(N^{-k/2}\) for the dual cumulants. This iteration process is one of the main improvements and novelties in the paper as it shows how to capture higher-order cumulant structures of the \(N\)-particle law, and not only the leading mean-field behavior itself. Finally, these estimates are then transferred through the weak duality framework of~\cite{bresch2024duality}, where the backward hierarchy is constructed in a suitable weighted Hilbert space and the forward \(N\)-particle evolution is encoded by weak duality solutions; the resulting pairing identity is precisely what relates the dual cumulants to the marginals. For \(k=1\) and \(k=2\), this relation provides quantitative bounds on the convergence of the marginals to the mean-field limit. For higher \(k\), the bounds on the dual cumulants can still be converted into corresponding bounds on the direct cumulants through the relation between the two notions.

We also note that our results are valid only for some fixed time interval of order $1$. This is not unusual with methods loosely based on Cauchy-Kovalevski arguments, such as is the case here to control a hierarchy of differential equations. A natural question is whether these local-in-time optimal estimates can be extended to arbitrary times. Unfortunately, this does not appear to be straightforward in the present case, and we do not have for example the equivalent of Lemma~2.1 in~\cite{duerinckx2025correlation} as the structure of the bounds that we can propagate is different.

\section{Dual Cumulants and Mean-Field Reformulation}\label{sec2}

In this section, we revisit the framework developed in \cite{bresch2024duality}, with particular emphasis on the notion of dual cumulants and their relation to the mean--field limit. Only the identities and equations needed in the present work are stated here; further details and proofs can be found in \cite{bresch2024duality}.

\medskip

\noindent
The following proposition was established in \cite{bresch2024duality}. We state it here in a form suitable for our purposes and provide a brief indication of the proof strategy, referring to Proposition 4 in \cite{bresch2024duality} for the complete proof.
\medskip

\begin{proposition}
Given a global weak duality solution $F_N$ of the Liouville equation~\eqref{eq:Liouville_equation}, as in Appendix A, with \(f^0\)--chaotic initial data~\eqref{eq:chaotic}, \(T >0\), \(m \geq 1\), and \(\psi \in C_c^{\infty}(\mathcal{D})\). Let \(h_N \in L^{\infty}\left([0,T] \times \mathcal{D}^N \right)\) be a bounded weak solution to the backward Liouville equation~\eqref{eq:backward_liouville}, with final data 

\begin{equation}\label{final_data_on_h_N}
    h_N(z_1,...,z_N)|_{t=T} = {N \choose m}^{-1} \sum_{1 \leq i_1 < \cdots<i_m \leq N} \psi(z_{i_1})\cdots \psi(z_{i_m}) \ .
\end{equation}
In addition, assume that the mean--field solution \(f\) satisfies \(K * f \in L^{\infty} \left([0,T] \times \Omega \right)\) and \(\nabla_v f \in L^1 \left([0,T] \times \mathcal{D} \right)\). Then the following identity holds 

\begin{equation}\label{prop_identity}
    \int_{\mathcal{D}^m} \psi^{\otimes m} \left(F_{N,m}(T) - f(T)^{\otimes m} \right) = -N \int_0^T \left(\int_{\mathcal{D}^N} V_f(z_1,z_2)h_N f^{\otimes N} \right)dt \ ,
\end{equation}
where we denote 

\begin{equation}\label{V_f_def}
V_{f}(z,z')
= \big( K(x-x') - K \ast f(x) \big)
\cdot \nabla_{v} \log f(z).
\end{equation}
\end{proposition}

\begin{proof}
Using the duality relation between \(h_N\) and \(F_N\) along with the final condition~\eqref{final_data_on_h_N} for \(h_N\), and the \(f^0\)-chaotic initial data~\eqref{eq:chaotic} for \(F_N\), also recalling that $h_N$ is symmetric in its $N$ variables, and using an approximation argument for Vlasov equation~\eqref{eq:vlasov_eq} and the backward Liouville equation~\eqref{eq:backward_liouville} for \(h_N\) yields the result.
\end{proof}

We also notice that by definition~\eqref{V_f_def}, \(V_f\) satisfies the following cancellation property 
\begin{equation}\label{cancellation_prop_V_f}
\int_{\mathcal{D}} V_f(z_1,z_2)\, f(z_j)\, dz_j = 0,
\qquad j=1,2.
\end{equation}

Following the dual correlation expansion introduced in \cite{bresch2024duality}, one can decompose the observable $h_N$ into contributions depending on subsets of particles, uniquely characterized by cancellation conditions with respect to the density $f$. We adopt this structure and reformulate it in terms of dual cumulants.

We now define the dual cumulants associated with an $N$--particle observable. Let $T>0$ and let $h_N \in L^\infty\big([0,T]\times \mathcal D^N\big).$
Assume that $f$ is a weak solution of the mean--field Vlasov equation~\eqref{eq:vlasov_eq}. The dual cumulants associated with $h_N$ are denoted by
\[
\big\{C_{N,n}\big\}_{0\le n\le N},
\]
where for each $n$, $C_{N,n}$ is a function of $n$ variables, with the convention $C_{N,n} \equiv 0$ for $n >N$.

The defining property of the dual cumulants is that they allow one to reconstruct $h_N$ as a sum over all subconfigurations of particles. More precisely, $h_N$ admits the decomposition
\begin{equation}\label{eq:dual_cumulant_expansion}
h_N(z_1,\dots,z_N)
=
\sum_{n=0}^N \ \sum_{\sigma\in\mathcal P_n^N}
C_{N,n}(z_\sigma),
\end{equation}
where $\mathcal P_n^N$ denotes the collection of all subsets of $\{1,\dots,N\}$ with $n$ elements, and $z_\sigma=(z_i)_{i\in\sigma}\in\mathcal D^n$. 

Since $h_N$ is bounded and symmetric in its $N$ variables, the dual cumulants $C_{N,n}$ are bounded and symmetric in their $n$ variables. By symmetry, the ordering of the variables in $z_\sigma$ is irrelevant.

The decomposition \eqref{eq:dual_cumulant_expansion} is not unique in general. To fix this ambiguity, we impose a cancellation condition with respect to the density $f$. Specifically, for every $n\ge1$, the dual cumulant $C_{N,n}$ is required to satisfy
\begin{equation}\label{eq:dual_cumulant_cancellation}
\int_{\mathcal D} C_{N,n}(z_1,\dots,z_n)\, f(z_j)\,dz_j = 0,
\qquad \forall\, 1\le j\le n.
\end{equation}

By definitions~\eqref{eq:dual_cumulant_expansion} and~\eqref{eq:dual_cumulant_cancellation} above and the cancellation property~\eqref{cancellation_prop_V_f} of \(V_f\), only the second--order dual cumulant contributes to the relation with \(h_N\), and we obtain the following identity
\begin{equation}\label{h_N_C_N,n_relation}
N\int_0^T \!\left( \int_{\mathcal{D}^N} V_f(z_1,z_2)\, h_N\, f^{\otimes N} \right) dt
=
N\int_0^T \!\left( \int_{\mathcal{D}^2} V_f(z_1,z_2)\, C_{N,2}\, f^{\otimes 2} \right) dt.
\end{equation}

Combining \eqref{prop_identity} and \eqref{h_N_C_N,n_relation} yields the following relation between propagation of chaos and the second--order dual cumulant

\begin{equation}\label{prop_cumulants}
     \int_{\mathcal{D}^m} \psi^{\otimes m} \left(F_{N,m}(T) - f(T)^{\otimes m} \right) = -N\int_0^T \!\left( \int_{\mathcal{D}^2} V_f(z_1,z_2)\, C_{N,2}\, f^{\otimes 2} \right) dt.
\end{equation}

\subsection{Truncated Rescaled Hierarchy}

Following the previous construction, we consider a bounded weak solution $h_N \in L^{\infty}([0,T] \times \mathcal{D}^{N})$ to the backward Liouville equation~\eqref{eq:backward_liouville}, taken in duality with $F_{N}$. Let $\{C_{N,n}\}_{0 \le n \le N}$ denote the associated dual cumulants introduced above.

In order to derive quantitative estimates, we introduce a rescaled version of these cumulants, denoted by $\bar{C}_{N,n}$, defined by
\[
\bar{C}_{N,n} := {N \choose n}^{\frac{1}{2}} C_{N,n},
\]
following the normalization introduced in \cite{bresch2024duality}.
These rescaled quantities satisfy a closed hierarchy. 
\medskip

\noindent

We now recall the rescaled hierarchy framework developed in \cite{bresch2024duality}, satisfied by $\bar{C}_{N,n}$ in truncated form. Here, the truncation consists in isolating the leading interaction terms, while grouping the remaining contributions into a remainder term $R_{N,n}$. We use the notation $z_{[n]} := (z_1,\dots,z_n)$.
\begin{align}\label{eq:hierarchy_equation}
&\partial_{t} \bar{C}_{N,n} - L_{n} \bar{C}_{N,n} = R_{N,n} 
+ \sum_{j=1}^{n} \int_{\mathcal{D}} V_{f}(z_{\star}, z_{j})\, \bar{C}_{N,n}(z_{[n] \setminus \{j\}}, z_{\star})\, f(z_{\star})\, dz_{\star} \notag \\
&\quad + \sqrt{(n+1)(n+2)} \int_{\mathcal{D}^2} V_{f}(z_{n+1}, z_{n+2})\, \bar{C}_{N,n+2}(z_{[n+2]})\, f(z_{n+1}) f(z_{n+2})\, dz_{n+1} dz_{n+2},
\end{align}
where \( L_n \) is the transport operator associated with the mean-field flow, defined by
\begin{align}\label{eq:operator_definition}
L_n g := -\sum_{i=1}^{n} \left( v_i \cdot \nabla_{x_i} + (K \ast f)(x_i) \cdot \nabla_{v_i} \right) g
,\end{align}
and the remainder term $R_{N,n}$ is given by
\begin{align}\label{remainder}
R_{N,n} :=\;& 
\frac{1}{N-1} \left( \frac{N - n + 1}{n} \right)^{1/2} S_{N}^{n,+} \bar{C}_{N,n-1} 
+ \tilde{S}_{N}^{n,\circ} \bar{C}_{N,n} + \left( \frac{n + 1}{N - n} \right)^{1/2} S_{N}^{n,-} \bar{C}_{N,n+1} \notag \\
&+ \sqrt{(n+1)(n+2)} 
\left( \left( \frac{(N - n)(N - n - 1)}{(N - 1)^2} \right)^{1/2} - 1 \right) \notag \\
&\quad \times \int_{\mathcal{D}^2} V_f(z_{n+1}, z_{n+2})\, \bar{C}_{N,n+2}(z_{[n+2]})\, f(z_{n+1})\, f(z_{n+2})\, dz_{n+1} dz_{n+2},
\end{align}
where we denote
\begin{align*}
S_N^{n,+} \bar{C}_{N,n-1}
:={}&
\sum_{i\ne j}^n (K\ast f)(x_i)\cdot\nabla_{v_i}\bar{C}_{N,n-1}(z_{[n]\setminus\{j\}})
\\
&-\sum_{i\ne j}^n
K(x_i-x_j)\cdot\nabla_{v_i}
\bar{C}_{N,n-1}(z_{[n]\setminus\{j\}}),
\\
&-\sum_{i\ne j}^n
\int_{\mathcal{D}} V_f(z_{\star},z_j)\,
\bar{C}_{N,n-1}(z_{[n]\setminus\{i,j\}},z_{\star})\,
f(z_{\star})\,dz_{\star},
\\
\end{align*}
while
\begin{align*}
\tilde{S}_N^{n,\circ} \bar{C}_{N,n}
:={}&
\frac{n-1}{N-1}
\sum_{i=1}^n (K\ast f)(x_i)\cdot\nabla_{v_i}\bar{C}_{N,n}
-\frac{1}{N-1}
\sum_{i\ne j}^n
K(x_i-x_j)\cdot\nabla_{v_i}\bar{C}_{N,n}
\\
&-\frac{n-1}{N-1}
\sum_{j=1}^n
\int_{\mathcal{D}} V_f(z_{\star},z_j)\,
\bar{C}_{N,n}(z_{[n]\setminus\{j\}},z_{\star})\,
f(z_{\star})\,dz_{\star}
\\
&+\frac{1}{N-1}
\sum_{i\ne j}^n
\int_{\mathcal{D}}
K(x_i-x_{\star})\cdot\nabla_{v_i}
\bar{C}_{N,n}(z_{[n]\setminus\{j\}},z_{\star})\,
f(z_{\star})\,dz_{\star}
\\
&-\frac{1}{N-1}
\sum_{i\ne j}^n
\int_{\mathcal{D}}
V_f(z_{\star},z_j)\,
\bar{C}_{N,n}(z_{[n]\setminus\{i\}},z_{\star})\,
f(z_{\star})\,dz_{\star}
\\
&+\frac{1}{N-1}
\sum_{i\ne j}^n
\int_{\mathcal{D}^2}
V_f(z_{\star},z_{\star}')\,
\bar{C}_{N,n}(z_{[n]\setminus\{i,j\}},z_{\star},z_{\star}')\,
f(z_{\star})f(z_{\star}')\,dz_{\star}dz_{\star}',
\end{align*}
and finally
\begin{align*}
S_N^{n,-} \bar{C}_{N,n+1}
:={}&
\frac{N-n}{N-1}
\sum_{j=1}^n
\int_{\mathcal{D}}
V_f(z_{\star},z_j)\,
\bar{C}_{N,n+1}(z_{[n]},z_{\star})\,
f(z_{\star})\,dz_{\star}
\\
&-\frac{N-n}{N-1}
\sum_{i=1}^n
\int_{\mathcal{D}}
K(x_i-x_{\star})\cdot\nabla_{v_i}
\bar{C}_{N,n+1}(z_{[n]},z_{\star})\,
f(z_{\star})\,dz_{\star}
\\
&-2\frac{N-n}{N-1}
\sum_{i=1}^n
\int_{\mathcal{D}^2}
V_f(z_{\star},z_{\star}')\,
\bar{C}_{N,n+1}(z_{[n]\setminus\{i\}},z_{\star},z_{\star}')\,
f(z_{\star})f(z_{\star}')\,dz_{\star}dz_{\star}'.
\end{align*}

\section{Quantitative Estimates for Dual Cumulants}\label{sec3}
This section is devoted to deriving estimates on the components of the rescaled dual cumulants introduced in the previous section, as stated in the following theorem.

\begin{theorem}\label{thm2}
Let \(k \geq 1\), \( K \in \mathscr{R}_k(\Omega; \mathbb{R}^d)
\cap H^{s}(\Omega; \mathbb{R}^d) \) where \( k - 1 < s \leq k \) and 
\[
\mathscr{R}_k(\Omega; \mathbb{R}^d)
 :=
\begin{cases}
L^{\infty}(\Omega; \mathbb{R}^d) + L^4(\Omega; \mathbb{R}^d) & \text{if } k = 1, \\
W^{k-1,\infty}(\Omega; \mathbb{R}^d) & \text{if } k \geq 2.
\end{cases}
\] 
Assume further that $f \in L^{\infty}(0, T; L^1(\mathcal{D}) \cap L^{\infty}(\mathcal{D}))$, and satisfies
\[
\int_{0}^{T} \left( \int_{\mathcal{D}} \Big( |\nabla_{v} \log f|^{2} + \sum_{j=1}^k \Big| \tfrac{1}{f} \nabla^{j}_{z} \nabla_{v} f \Big|^{2} \Big) f \right)^{\frac{1}{2}} < \infty,
\qquad
K * f \in L^{\infty}\left(0, T; W^{k+1,\infty}(\Omega; \mathbb{R}^d)\right),
\]
along with $\operatorname{ess}\sup_{x_\star} \int |\nabla_{v_\star}\log f(z_\star)|^2 f(z_\star)\,dv_\star < \infty$.
Assume that, at the final time $T$, the rescaled dual cumulants satisfy the following estimates for some $0\leq k_0\leq k$
\begin{equation}
\|\bar C_{N,n}(T)\|_{L^2_f}\leq C^n\,N^{-k_0/2}\quad \mbox{for}\ n\leq k_0,\quad\|\bar C_{N,n}(T)\|_{L^2_f}\leq C^n\,N^{-n/2}\quad \mbox{for}\ n\geq k_0.\label{finaldatadual}
\end{equation}
Then, there exists a decomposition of the dual cumulants \(\bar{C}_{N,n}\) as
\[\bar{C}_{N,n}:=h^{(k)}_{N,n} +\sum_{q=1}^k\sum_{i_1,...,i_q=1}^n div_{z_{i_1}...z_{i_q}} w^{(k),i_1,...,i_q}_{N,n},
\]
and a time $\Delta t_k > 0$, independent of $N$, such that for all \(1 \leq q \leq k-1\), \( k \leq n \leq N \) and for all \(t \in [T-\Delta t_k,T]\), we obtain the bounds
\begin{equation}\label{eq:cumulants_bound0}
 \|w_{N,n}^{(k),i_1,...,i_k}(t) \|_{L^2_f} \lesssim C_k C^{nk} N^{-\frac{k}{2}} 
\end{equation}
\begin{equation}\label{eq:cumulants_bound02}
\|w_{N,n}^{(k),i_1,...,i_q}(t) \|_{L^2_f} \lesssim C_k C^{nk} N^{-\frac{k}{2}} 
\end{equation}
\begin{equation}\label{eq:cumulants_bound03}
\|h_{N,n}^{(k)}(t) \|_{L^2_f} \lesssim C_k C^{nk} N^{-\frac{s}{2}}, 
\end{equation}
for some $C$ independent of $k$ and $N$ and some ${C}_k$ behaving at least exponentially. For $n<k$, we have the simpler bounds
\begin{equation}
  \|w_{N,n}^{(k),i_1,...,i_k}(t) \|_{L^2_f}+  \|w_{N,n}^{(k),i_1,...,i_q}(t) \|_{L^2_f}+\|h_{N,n}^{(k)}(t) \|_{L^2_f}\lesssim  C_k C^{nk}\,N^{-\max(k_0,n)/2}.\label{eq:cumulants_boundk0}
\end{equation}
\end{theorem}
In general we expect the rescaled dual cumulant $\bar C_{N,n}$ to behave like $C^n\,N^{-n/2}$ at the final time, or the non-rescaled dual cumulant $C_{N,n}$ to behave like $C^n\,N^{-n}$. This corresponds to~\eqref{finaldatadual} with $k_0=0$. However in some special cases, the first few dual cumulants may be much smaller, which is why we can allow for $k_0>0$.

\subsection{Decomposing the Mean-Field Flow}
The goal is to decompose \(\bar C_{N,n}\) into components carrying different numbers of exterior divergence operators and derive a separate equation for each component.

Throughout this section, we make use of the following decomposition of the interaction operator \(V_f\) previously defined in~\eqref{V_f_def}
\begin{equation}\label{eq:V_f_decomposition}
   V_f = V_f^{\delta} + W_f^{\delta}, 
\end{equation}
with the components given explicitly by
\[
V_{f}^{\delta}(z,z')
= \big( K_{\delta}(x-x') - K \ast f(x) \big)
\cdot \nabla_{v} \log f(z),
\]
and
\[
W_{f}^{\delta}(z,z')
= \big( K(x-x') - K_{\delta}(x-x') \big)
\cdot \nabla_{v} \log f(z).
\]
Here the regularized kernel \(K_{\delta}\) is defined by
\[
K_{\delta} := K \ast \rho_{\delta},
\quad
\rho_{\delta}(x) := \delta^{-d}\,\rho\!\left(\frac{x}{\delta}\right),
\]
for some parameter $\delta >0$ to be properly chosen later on, and for a smooth mollifier \(\rho \in C_c^\infty(\mathbb R^d)\) with compact support.

We also introduce the evolution operator associated with the
\(n\)-particle mean-field transport operator \(L_n\).
For each $n\ge1$, let $G^{(n)}_{t,T}$ denote the operator acting on functions of
$z_{[n]}=(z_1,\dots,z_n)\in\mathcal D^n$, generated by the operator $L_n$ defined in~\eqref{eq:operator_definition}. Namely the solution $F_{n,i}$ to
\[
\partial_t F_{n,i}-L_n\,F_{n,i}=g_{n,i},\quad F_{n,i}(t=T)=\bar F_{n,i}
\]
is given by
\[
F_{n,i}(t,z_{[n]})=G^{(n)}_{t,T}\,\bar F_{n,i}-\int_t^T G^{(n)}_{t,\tau}\,g_{n,i}(\tau)\,d\tau.
\]
Note that $G^{(n)}_{t,T}$ is trivially bounded on every $L^p$ space with an operator norm that is uniform in $t$ and $T$.
\medskip

We now observe that, for each fixed $n$, the truncated hierarchy for $\bar C_{N,n}$ can be viewed as a linear transport equation of the form
\[
\partial_t \bar C_{N,n} - L_n \bar C_{N,n} = g_n,
\]
where the source term $g_n$ collects the coupling terms and the remainder $R_{N,n}$. In particular, $g_n$ admits a decomposition in terms of derivatives of lower- and higher-order cumulants. Because our analysis involves right-hand sides with derivatives, we must also explain how derivatives commute with \(G^{(n)}\).
This is provided by the following lemma.
\begin{lemma} \label{lem:one_particle_k} 
Fix $i\in\{1,\dots,n\}$ and $k\ge 1$.
Let
\[
A_n^{(i)} := v_i\cdot\nabla_{x_i} + (K*f)(x_i)\cdot\nabla_{v_i},
\]
with $K \ast f \in L^{\infty}\left(0,T ;W^{k+1,\infty}(\Omega ; \mathbb{R}^d) \right)$, and denote by $G^{(n,i)}_{t,T}$ the corresponding operator solving the PDE: $( \partial_t+A_n^{(i)} )u=0$.
Then, there exists a family of operators $G^{(n,i)}$ and $\{G^{(n,i),\ell,k}\}_{\ell=0}^k$
bounded over $L^2$ with an operator norm that depends only on $k$ and $T-t$
such that 
\[
G_{t,T}^{(n,i)}\nabla^k_{z_i} g_{n,i}=\sum_{\ell =0}^k\nabla_{z_i}^{\,\ell}
G^{(n,i),\ell,k}_{t,T} g_{n,i}.
\]
\end{lemma}

\begin{proof}
Note that all variables $x_j$ and $v_j$ for $j\neq i$ are just parameters so that we will not indicate them in the various formulas. Using the method of characteristics, we define the characteristic curves
\[
\dot{X}_i(t,s,x,v) = V_i(t,s,x,v),
\qquad
X_i(t=s,s,x,v)=x_i,
\]
and
\[
\dot{V}_i(t,s,x,v) = K * f \bigl(X_i(t,s,x,v)\bigr),
\qquad
V_i(t=s,s,x,v) = v_i.
\]
Define
\[
F_{n,i}(t) := f_{n,i}\bigl(t, X_i(t,s,x,v), V_i(t,s,x,v)\bigr).
\]
By the chain rule, we compute
\begin{align*}
\partial_t F_{n,i}
&=
\partial_t f_{n,i}(t,X_i,V_i)
+
\bigl(\dot X_i\cdot\nabla_{x_i}
+
\dot V_i\cdot\nabla_{v_i}\bigr)
f_{n,i}(t,X_i,V_i) \\
&=
\partial_t f_{n,i}(t,X_i,V_i)
+
A_n^{(i)} f_{n,i}(t,X_i,V_i) \\
&=
\nabla_{z_i}^k g_{n,i}(t,X_i,V_i).
\end{align*}
Integrating from $t$ to $T$ yields
\[
F_{n,i}(t)
=
F_{n,i}(T)
-
\int_t^T
\nabla_{z_i}^k g_{n,i}\bigl(\tau, X_i(\tau,s,x,v), V_i(\tau,s,x,v)\bigr)\,d\tau.
\]
Taking $s=t$, we obtain the representation of the operator $G^{(n,i)}_{t,T}$ as
\[
G^{(n,i)}_{t,T} \bar f_{n,i}=\bar f_{n,i}(X_i(T,t,x,v),V_i(T,t,x,v)),
\]
so that $G^{(n,i)}_{t,\tau}$ is naturally bounded on $L^2$.

\noindent
We now need to rewrite $G^{(n,i)}_{t,\tau}\nabla_{z_i}^k$ as a finite sum of
derivatives acting outside the operator, and we proceed by induction on $k$.

\medskip
\noindent\textbf{Step 1: Case $k=1$.}
Assume that $g_{n,i}$ is sufficiently smooth and denote by
\[
Z_i(\tau,t,x,v) := (X_i(\tau,t,x,v), V_i(\tau,t,x,v))
\]
the characteristic flow associated with the one--particle operator $A_n^{(i)}$.
By the definition of the backward flow operator, for any test function $h$ we have
\[
\big(G^{(n,i)}_{t,\tau} h\big)(x,v) = h\big(\tau, Z_i(\tau,t,x,v)\big).
\]
In particular,
\[
G^{(n,i)}_{t,\tau}\nabla_{z_i} g_{n,i}
=
\big(\nabla_{z_i} g_{n,i}\big)\big(\tau,Z_i(\tau,t,x,v)\big).
\]
The key observation is that, since the right--hand side of the equation already
contains one derivative, we do not aim to eliminate derivatives altogether, but
rather to rewrite this expression so that exactly one derivative appears
\emph{outside} an operator acting on $g_{n,i}$. By the chain rule,
\[
\nabla_{z_i}\!\left(g_{n,i}\big(\tau,Z_i\big)\right)
=
\big(\nabla_{z_i} g_{n,i}\big)\big(\tau,Z_i\big)\,\nabla Z_i,
\]
where $\nabla Z_i$ denotes the Jacobian matrix of the flow with respect to $z_i$.
Since the characteristic flow is a diffeomorphism, the matrix $\nabla Z_i$ is
invertible, and we may therefore write
\[
\big(\nabla_{z_i} g_{n,i}\big)\big(\tau,Z_i\big)
=
(\nabla Z_i)^{-1}
\nabla_{z_i}\!\left(g_{n,i}\big(\tau,Z_i\big)\right).
\]
Applying the product rule to the right--hand side yields the decomposition
\[
\big(\nabla_{z_i} g_{n,i}\big)\big(\tau,Z_i\big)
=
\nabla_{z_i}\!\left((\nabla Z_i)^{-1} g_{n,i}\big(\tau,Z_i\big)\right)
-
\nabla_{z_i}(\nabla Z_i)^{-1}\,
g_{n,i}\big(\tau,Z_i\big).
\]
This naturally leads us to define the operators
\[
G^{(n,i),1,1}_{t,\tau} g_{n,i}
:=
(\nabla Z_i)^{-1}\, g_{n,i}\big(\tau,Z_i\big),
\qquad
G^{(n,i),0,1}_{t,\tau} g_{n,i}
:=
-\nabla_{z_i}(\nabla Z_i)^{-1}\, g_{n,i}\big(\tau,Z_i\big),
\]
so that
\begin{equation}\label{eq2}
G^{(n,i)}_{t,\tau}\nabla_{z_i} g_{n,i}
=
\nabla_{z_i}\big(G^{(n,i),1,1}_{t,\tau} g_{n,i}\big)
+
G^{(n,i),0,1}_{t,\tau} g_{n,i}.
\end{equation}
The matrices $(\nabla Z_i)^{-1}$ and $\nabla_{z_i}(\nabla Z_i)^{-1}$ depend only on
the characteristic flow. Under the assumption that $K*f \in L^{\infty}(0,T;W^{2,\infty}(\Omega;\mathbb{R}^d))$, standard Gr\"onwall estimates for the
Jacobian of the flow imply that the coefficients $(\nabla Z_i)^{-1}$ and
$\nabla_{z_i}(\nabla Z_i)^{-1}$ are bounded uniformly in space on any finite time
interval. Consequently, for $\ell=0,1$, all $0\le t\le \tau\le T$, and all $h \in L^2$, the operators
$G^{(n,i),\ell,1}_{t,\tau}$ are bounded on $L^2$, with
\[
\|G^{(n,i),\ell,1}_{t,\tau} h\|_{L^2}
\le C_1\,e^{C_1\,(T-t)}\|h \|_{L^2}.
\]
Inserting identity \eqref{eq2} into the solution representation with $k=1$
yields the desired result.

\medskip
\noindent\textbf{Step 2: General case.}
Let $G^{(n,i)}_{t,\tau}\nabla_{z_i}^{k} g_{n,i}$ for a sufficiently smooth $g_{n,i}$. We prove the corresponding statement at order $k$.
Applying the $k=1$ decomposition of Step~1 to the function
$h:=\nabla_{z_i}^{k-1}g_{n,i}$ gives
\begin{equation}\label{eq:step1_applied_to_h}
G^{(n,i)}_{t,\tau}\nabla_{z_i}h
=
\nabla_{z_i}\big(G^{(n,i),1,1}_{t,\tau}h\big)
+
G^{(n,i),0,1}_{t,\tau}h,
\end{equation}
hence
\[
G^{(n,i)}_{t,\tau}\nabla_{z_i}^{k} g_{n,i}
=
\nabla_{z_i}\Big(G^{(n,i),1,1}_{t,\tau}\nabla_{z_i}^{k-1}g_{n,i}\Big)
+
G^{(n,i),0,1}_{t,\tau}\nabla_{z_i}^{k-1}g_{n,i}.
\]
We now repeatedly apply the first-order decomposition from Step~1 to move
derivatives outside the flow operators. Since each application produces only
bounded coefficient operators coming from the characteristic flow and lowers the
number of derivatives falling on $g_{n,i}$ by one, after finitely many iterations
we may expand and re-index the resulting finite sum to obtain
\[
G^{(n,i)}_{t,\tau}\nabla_{z_i}^{k} g_{n,i}
=
\sum_{\ell=0}^{k}
\nabla_{z_i}^{\,\ell}\big( G^{(n,i),\ell,k}_{t,\tau} g_{n,i}\big),
\]
for a suitable family of bounded operators
$\{G^{(n,i),\ell,k}_{t,\tau}\}_{\ell=0}^k$.

\noindent
Finally, by the same reasoning as in Step 1, for $0 \leq \ell \leq k$, each $G^{(n,i),\ell,k}_{t,\tau}$ is bounded
on $L^2$, and we may write
\[
\|G^{(n,i),\ell,k}_{t,\tau} h\|_{L^2}
\le C_k(T-t)\,\|h\|_{L^2},
\qquad \forall\, h\in L^2.
\]
This completes the induction step.
\end{proof}

We now write the PDE and the solution representation for the fully tensorized $n$--particle flow with derivatives of order exactly $k$ acting across all particle variables. This then yields the corresponding representation with derivatives of order up to $k$.
\begin{lemma}
\label{lem:tensor_all_derivatives}
Let
\[
A_n:=\sum_{i=1}^n A_n^{(i)}, \qquad L_n:=-A_n,
\qquad
z_{[n]}=(z_1,\dots,z_n).
\]
Assume that the source term admits the decomposition
\[
g_n(t,z_{[n]})
=
\sum_{r=0}^k \ \sum_{i_1,\dots,i_r=1}^n
\nabla_{z_{i_1}\cdots z_{i_r}}\,
g_n^{(k),i_1,\dots,i_r}(t,z_{[n]}).
\]
Consider the terminal value problem
\[
\partial_t f_n + A_n f_n = g_n,
\qquad
f_n(T)=\bar f_n.
\]
For $i_{[r]}=(i_1,\dots,i_r)\in\{1,\dots,n\}^r$, define
\[
S(i_1,\dots,i_r)
:=
\Big\{
s\in\mathbb{N}^n :
\operatorname{supp}(s)\subset \{i_1,\dots,i_r\},
\ |s|\le r
\Big\}.
\]
Then there exist operators
\[
G^{(n)}_{t,T}
\quad\text{and}\quad
\big\{ G^{(n),s,i_{[r]}}_{t,\tau} \big\}_{r=|s|,\dots,k},
\]
bounded on $L^2(\mathcal D^n)$, such that the solution admits the representation
\[
f_n(t)
=
G^{(n)}_{t,T}\bar f_n
-
\int_t^T
\sum_{r=0}^k
\sum_{i_1,\dots,i_r=1}^n
\sum_{s\in S(i_1,\dots,i_r)}
\nabla^{\,s}\Big(
G^{(n),s,i_{[r]}}_{t,\tau}\,
g_n^{(k),i_1,\dots,i_r}(\tau)
\Big)\,d\tau.
\]
Moreover, there exists a constant
$C_k$ such that for all $(i_1,\dots,i_r)$, all
$s\in S(i_1,\dots,i_r)$, and all $h\in L^2(\mathcal D^n)$
\begin{equation}\label{Operators_bound_tensorized}
\|G^{(n),s,i_{[r]}}_{t,\tau} h\|_{L^2(\mathcal D^n)}
\le C_k(T-t)\,\|h\|_{L^2(\mathcal D^n)}.
\end{equation}
\end{lemma}
\begin{proof}
Each one--particle operator $A_n^{(i)}$ generates a backward flow
$G^{(n,i)}_{t,\tau}$ on the variables $(x_i,v_i)$. Since the full operator
$A_n=\sum_{i=1}^n A_n^{(i)}$ is tensorized, the associated $n$--particle backward
flow is given by
\[
G^{(n)}_{t,\tau}:=\bigotimes_{i=1}^n G^{(n,i)}_{t,\tau}.
\]
A direct computation shows that $G^{(n)}_{t,\tau}$ solves
$\partial_t u + A_n u = 0$ on $\mathcal D^n$, and therefore Duhamel’s formula yields
\[
f_n(t)
=
G^{(n)}_{t,T}\bar f_n
-
\int_t^T G^{(n)}_{t,\tau} g_n(\tau)\,d\tau.
\]
Substituting the decomposition of $g_n$ and using linearity, we obtain
\[
f_n(t)
=
G^{(n)}_{t,T}\bar f_n
-
\int_t^T
\sum_{r=0}^k
\sum_{i_1,\dots,i_r=1}^n
G^{(n)}_{t,\tau}
\big(
\nabla_{z_{i_1}\cdots z_{i_r}}\, g_n^{(k),i_1,\dots,i_r}(\tau)
\big)\,d\tau.
\]
Fix $r$ and a tuple $(i_1,\dots,i_r)$. Since $G^{(n)}_{t,\tau}$ is tensorized, we
may apply Lemma~\ref{lem:one_particle_k} independently in each of the variables
$z_{i_1},\dots,z_{i_r}$, while the remaining factors act only by the (derivative--free)
one--particle flow. Iterating Lemma~\ref{lem:one_particle_k} yields
$L^2$ bounded operators $G^{(n),s,i_{[r]}}_{t,\tau}$ such that
\[
G^{(n)}_{t,\tau}
\big(
\nabla_{z_{i_1}\cdots z_{i_r}}\, h
\big)
=
\sum_{s\in S(i_1,\dots,i_r)}
\nabla^{\,s}\Big(
G^{(n),s,i_{[r]}}_{t,\tau}\, h
\Big).
\]
Applying this identity with
$h=g_n^{(k),i_1,\dots,i_r}(\tau)$ and summing over $r$ and the indices
$(i_1,\dots,i_r)$ yields the claimed representation formula.

\noindent
Finally, by construction each operator $G^{(n),s,i_{[r]}}_{t,\tau}$ acts nontrivially
only on the particle variables indexed by $\{i_1,\dots,i_r\}$, and coincides with
the derivative--free one--particle flow on all remaining variables. Consequently,
$G^{(n),s,i_{[r]}}_{t,\tau}$ is obtained by composing and taking finite linear
combinations of the one--particle operators provided by
Lemma~\ref{lem:one_particle_k} in at most $0 \le r\le k$ variables.

\noindent
Using the $L^2$--boundedness of these one--particle operators and the fact that the  transport flow in the remaining variables preserves the $L^2$ norm, we obtain a
constant $C_k$ such that for all $0\le r\le k$, all tuples
$(i_1,\dots,i_r)$, all $s\in S(i_1,\dots,i_r)$, and all $h\in L^2(\mathcal D^n)$,
\[
\|G^{(n),s,i_{[r]}}_{t,\tau} h\|_{L^2(\mathcal D^n)}
\le C_k\,(T-t)\,\|h\|_{L^2(\mathcal D^n)}.
\]
Here \(C_k\) depends only on \(k\) and on the relevant bounds for
\(K*f\) on \([t,T]\). It may be chosen as the maximum of the constants
\(C_r\) in Lemma \ref{lem:one_particle_k} for \(0\le r\le k\). This concludes the proof.
\end{proof}

\subsection{First--Order Decomposition of the Dual Cumulants}
We now state and prove the following decomposition lemma for the dual cumulants. In this section, we use the notation $\Lambda$ for constants independent of $N,\,n,\,t,\,T$, the value of which may change from line to line.
\begin{lemma}\label{lem:decomposition-h1-w1}
There exists a decomposition of the \(\bar C_{N,n}\) as
\[
\bar{C}_{N,n}
= h^{(1)}_{N,n} + \sum_{i_1=1}^{n} \mathrm{div}_{z_{i_1}} w^{(1),i_1}_{N,n},
\qquad
\bar{C}_{N,n+2}
= h^{(1)}_{N,n+2} + \sum_{i_1=1}^{n+2} \mathrm{div}_{z_{i_1}} w^{(1),i_1}_{N,n+2},
\]
where the component \(h^{(1)}_{N,n}\) satisfies the following hierarchy
\[
\begin{aligned}
h^{(1)}_{N,n}(t)=&G_{t,T}^{(n)}\,\bar C_{N,n}(T)
 -\int_t^T G_{t,\tau}^{(n)} \Bigg(
\Lambda\!\left(\tfrac{n+1}{N}\right)^{\frac12} R^{(1)}_{N,n}
+H_{N,n}^{1,1}+H_{N,n}^{1,2} 
\Bigg)\,d\tau \\
&-\int_t^T \sum_{i_1=1}^n G_{t,\tau}^{(n),0,i_1} \Bigg(\Lambda\! \left( \tfrac{n+1}{N}\right)^{\frac{1}{2}}R_{N,n}^{(1),i_1} +  \sum_{j=1}^n \int \Gamma (z_\star,z_j) w^{(1),i_1}_{N,n} f(z_\star) dz_\star \\
&+ n \int V_f^{\delta} (z_{n+1}, z_{n+2}) w^{(1),i_1}_{N,n+2} f(z_{n+1}) f(z_{n+2}) \Bigg) \, d\tau.
\end{aligned}
\]
Here we denote
\[
\begin{split}
    H^{1,1}_{N,n}=&\sum_{j=1}^n \int V_f \, h^{(1)}_{N,n} \, f(z_{\star})\,dz_{\star} - \sum_{j=1}^n \int \nabla_{z_{\star}}
\Big( K \ast f(x_{\star}) \cdot \nabla_{v_{\star}} \log f(z_{\star})\, f(z_{\star}) \Big)
\, w^{(1),j}_{N,n}\,dz_{\star}\\
&+ \sum_{j=1}^n \int K(x_{\star}-x_j)
\Big(\nabla_{z_{\star}} \nabla_{v_{\star}} \log f(z_{\star})\, f(z_{\star})\Big)
\, w^{(1),j}_{N,n}\,dz_{\star}, 
\end{split}
\]
and 
\[
\begin{split}
    H^{1,2}_{N,n}=& n \!\!\int W_f^{\delta}\, \bar{C}_{N,n+2}\,
f(z_{n+1}) f(z_{n+2})\,dz_{n+1}dz_{n+2}\\
&+ n \!\!\int V_f^{\delta}\, h^{(1)}_{N,n+2}\,
f(z_{n+1}) f(z_{n+2})\,dz_{n+1}dz_{n+2} \\
&+ n \!\!\int \nabla_{z_{n+1}}(V_f^{\delta} f)\,
w_{N,n+2}^{(1),n+1}\, f(z_{n+2})\,dz_{n+1}dz_{n+2}\\
&+ n \!\!\int \nabla_{z_{n+2}}(V_f^{\delta} f)\,
w_{N,n+2}^{(1),n+2}\, f(z_{n+1})\,dz_{n+1}dz_{n+2}.
\end{split}
\]
The \(w^{(1),i_1}_{N,n}\) solve the corresponding hierarchy
\[
\begin{aligned}
\nabla_{z_{i_1}} w^{(1),i_1}_{N,n}(t)
= -\nabla_{z_{i_1}} \int_t^T G^{(n),i_1,i_1}_{t,\tau}\Big(
&\Lambda\!\left(\tfrac{n+1}{N}\right)^{\frac12} R^{(1),i_1}_{N,n}
+ \sum_{j=1}^n \int \Gamma(z_\star,z_j)\, w^{(1),i_1}_{N,n}\, f(z_\star)\,dz_\star \\
&+ n \int V_f^\delta(z_{n+1},z_{n+2})\,
w^{(1),i_1}_{N,n+2}\,
f(z_{n+1}) f(z_{n+2})
\Big)\,d\tau,
\end{aligned}
\]
where we denote
\[
\Gamma(z_\star,z_j)
= (1-\delta_{i_1j})V_f(z_\star,z_j)
- \delta_{i_1j} K(x_\star-x_j)\,
\nabla_{z_\star}\log f(z_\star).
\]
\end{lemma}
\begin{proof}
Denote
\[
g_n := R_{N,n} + \mathcal J[\bar C_{N,n}] + \bar{\mathcal J}[\bar C_{N,n+2}],
\]
where the interaction operators are given by
\[
\mathcal J[\bar C_{N,n}]
=
\sum_{j=1}^{n}
\int_{\mathcal D}
V_f(z_{\star}, z_j)\,
\bar C_{N,n}(z_{[n]\setminus\{j\}}, z_\star)\,
f(z_{\star})\,dz_{\star},
\]
and
\[
\begin{aligned}
\bar{\mathcal J}[\bar C_{N,n+2}]
&=\sqrt{(n+1)(n+2)}
\int_{\mathcal D^2}
V_f(z_{n+1},z_{n+2}) \\
&\qquad 
\bar C_{N,n+2}(z_{[n+2]})\,
f(z_{n+1})f(z_{n+2})\,
dz_{n+1}dz_{n+2}.
\end{aligned}
\]
Assume that the remainder can be decomposed as 
\[
R_{N,n}
= R_{N,n}^{(1)} + \sum_{i_1=1}^n \mathrm{div}_{z_{i_1}} R_{N,n}^{(1),i_1}
,\]
together with the corresponding decomposition for $\mathcal J[\bar C_{N,n}]$ and $\bar{\mathcal J}[\bar C_{N,n+2}]$,
\[\mathcal J[\bar C_{N,n}] = \mathcal J^{(1)}[\bar C_{N,n}] + \sum_{i_1=1}^n \nabla_{z_{i_1}} \mathcal J^{(1),i_1}[\bar C_{N,n}],\]
and 
\[
\bar{\mathcal J}[\bar C_{N,n+2}] = \bar{\mathcal J}^{(1)}[\bar C_{N,n+2}] + \sum_{i_1=1}^n \nabla_{z_{i_1}} \bar{\mathcal J}^{(1),i_1}[\bar C_{N,n+2}].
\]
The 3 decompositions above immediately yield the following decomposition for $g_n$
\[g_n = g_n^{(1)} + \sum_{i_1=1}^n \nabla_{z_{i_1}} g_n^{(1),i_1}, 
\]
where 
\[
g_n^{(1)} = R_{N,n}^{(1)} + \mathcal J^{(1)}[\bar C_{N,n}] + \bar{\mathcal J}^{(1)}[\bar C_{N,n+2}],
\]
and 
\[
g_n^{(1),i_1} = R_{N,n}^{(1),i_1} + \mathcal J^{(1),i_1}[\bar C_{N,n}] + \bar{\mathcal J}^{(1),i_1}[\bar C_{N,n+2}].
\]
Then, by Lemma \ref{lem:tensor_all_derivatives} for the case $k=1$, the solution to the truncated hierarchy \eqref{eq:hierarchy_equation} has the following form 
\begin{equation}\label{eq:Duhamel_rep}
\begin{aligned}
\bar{C}_{N,n}(t)
&= G_{t,T}^{(n)}\bar{C}_{N,n}(T)
- \int_t^T \Bigg(
G_{t,\tau}^{(n)}g_n^{(1)}(\tau)
+ \sum_{i_1=1}^n G_{t,\tau}^{(n),0,i_1}g_n^{(1),i_1}(\tau) \\
&\qquad\qquad
+ \sum_{i_1=1}^n
\nabla_{z_{i_1}}G_{t,\tau}^{(n),i_1,i_1}
g_n^{(1),i_1}(\tau)
\Bigg)\,d\tau ,
\end{aligned}
\end{equation}
for some $L^2$ bounded operators $ G_{t,\tau}^{(n)} \ , \ G_{t,\tau}^{(n),0,i_1} \ , \ G_{t,\tau}^{(n),i_1,i_1}$. \\

\noindent
We then match the various terms in~\eqref{eq:Duhamel_rep} with the decompositions of \(V_f\) in~\eqref{eq:V_f_decomposition} and of $\bar C_{N,n}$: we collect all terms without derivatives
in the variables \(z_{i_1}\) into the equation for \(h^{(1)}_{N,n}\), and all terms
containing exactly one divergence \(\mathrm{div}_{z_{i_1}}\) into the equation
for \(w^{(1),i_1}_{N,n}\). We need to have special care for the term
\[
\sum_{j=1}^n \int \nabla_{z_\star}(V_f f)(z_\star,z_j)\,
w^{(1),j}_{N,n}(z_{[n]\setminus\{j\}}, z_\star)\,dz_\star,
\]
which can be rewritten by using the explicit formula for \(V_f\) and using the identity
\[
\nabla_{x_\star}K(x_\star-x_j)=-\nabla_{x_j}K(x_\star-x_j).
\]
This leads to the decomposition
\[
\int \nabla_{z_\star}(V_f f)(z_\star,z_j)\,
w^{(1),j}_{N,n}\,dz_\star
= I_1 - I_2,
\]
where
\[
I_2
=\int \nabla_{z_\star}
\!\left(K*f(x_\star)\,\nabla_{v_\star}\log f(z_\star)\,f(z_\star)\right)
w^{(1),j}_{N,n}(z_{[n]\setminus\{j\}}, z_\star)\,dz_\star,
\]
is controlled by the regularity of \(f\), and is therefore
assigned to the equation for \(h^{(1)}_{N,n}\).
On the other hand, we have
\[
I_1
=-\nabla_{z_j}
\int K(x_\star-x_j)\,
\nabla_{v_\star}\log f(z_\star)\,
f(z_\star)\,
w^{(1),j}_{N,n}(z_{[n]\setminus\{j\}}, z_\star)\,dz_\star,
\]
which exhibits a full derivative and contributes to the equation for \(w^{(1),i_1}_{N,n}\). This yields as claimed the following equation on $w_{N,n}^{(1),i_1}$
\[
\begin{aligned}
&w^{(1),i_1}_{N,n}(t)
= -\int_t^T G_{t,\tau}^{(n),i_1,i_1}\Big(
\Lambda\!\left(\tfrac{n+1}{N}\right)^{1/2}
R^{(1),i_1}_{N,n}
+ \sum_{j=1}^n \int \Gamma(z_\star,z_j)\,
w^{(1),i_1}_{N,n}(z_{[n]\setminus\{j\}}, z_\star)\,
f(z_\star)\,dz_\star \\
&\quad+ \sqrt{(n+1)(n+2)}\,
\int V_f^\delta(z_{n+1},z_{n+2})\,
w^{(1),i_1}_{N,n+2}(z_{[n+2]})
f(z_{n+1}) f(z_{n+2})\,dz_{n+1}dz_{n+2}\Big)\,d\tau.
\end{aligned}
\]
On the other hand $h^{(1)}_{N,n}$ solves
\[
\begin{aligned}
h^{(1)}_{N,n}(t)=&G_{t,T}^{(n)}\,\bar C_{N,n}(T)
 -\int_t^T G_{t,\tau}^{(n)} \Bigg(
\Lambda\!\left(\tfrac{n+1}{N}\right)^{\frac12} R^{(1)}_{N,n}
+H_{N,n}^{1,1}+H_{N,n}^{1,2} 
\Bigg)\,d\tau \\
&-\int_t^T \sum_{i_1=1}^n G_{t,\tau}^{(n),0,i_1} \Bigg(\Lambda\! \left( \tfrac{n+1}{N}\right)^{\frac{1}{2}}R_{N,n}^{(1),i_1} +  \sum_{j=1}^n \int \Gamma (z_\star,z_j) w^{(1),i_1}_{N,n} f(z_\star) dz_\star \\
&+ \sqrt{(n+1)(n+2)} \int V_f^{\delta} (z_{n+1}, z_{n+2}) w^{(1),i_1}_{N,n+2} f(z_{n+1}) f(z_{n+2})\,dz_{n+1}dz_{n+2} \Bigg) \, d\tau,
\end{aligned}
\]
where we denote
\[
\begin{split}
    H^{1,1}_{N,n}=&\sum_{j=1}^n \int V_f \, h^{(1)}_{N,n} \, f(z_{\star})\,dz_{\star} - \sum_{j=1}^n \int \nabla_{z_{\star}}
\Big( K \ast f(x_{\star}) \cdot \nabla_{v_{\star}} \log f(z_{\star})\, f(z_{\star}) \Big)
\, w^{(1),j}_{N,n}\,dz_{\star}\\
&+ \sum_{j=1}^n \int K(x_{\star}-x_j)
\Big(\nabla_{z_{\star}} \nabla_{v_{\star}} \log f(z_{\star})\, f(z_{\star})\Big)
\, w^{(1),j}_{N,n}\,dz_{\star}, 
\end{split}
\]
and 
\[
\begin{split}
    H^{1,2}_{N,n}=& \sqrt{(n+1)(n+2)} \Bigg( \!\!\int W_f^{\delta}\, \bar{C}_{N,n+2}\,
f(z_{n+1}) f(z_{n+2})\,dz_{n+1}dz_{n+2}\\
&+  \!\!\int V_f^{\delta}\, h^{(1)}_{N,n+2}\,
f(z_{n+1}) f(z_{n+2})\,dz_{n+1}dz_{n+2} \\
&+  \!\!\int \nabla_{z_{n+1}}(V_f^{\delta} f)\,
w_{N,n+2}^{(1),n+1}\, f(z_{n+2})\,dz_{n+1}dz_{n+2}\\
&+  \!\!\int \nabla_{z_{n+2}}(V_f^{\delta} f)\,
w_{N,n+2}^{(1),n+2}\, f(z_{n+1})\,dz_{n+1}dz_{n+2}\Bigg),
\end{split}
\]
which concludes the proof.
\end{proof}
We can now take advantage of the explicit equations for $h^{(1)}_{N,n}$ and $w_{N,n}^{(1),i_1}$ to derive explicit bounds on these quantities. This is established in the following proposition, which treats the case \(k=1\) of Theorem~\ref{thm2} and provides the base case for the induction argument.
\begin{proposition} \label{component_bounds}
Let $\bar C_{N,n}$ be decomposed according to Lemma~\ref{lem:decomposition-h1-w1}, and let \(K \in \left( L^{4}(\Omega;\mathbb{R}^d)+ L^{\infty}(\Omega;\mathbb{R}^d) \right)\cap H^{s}(\Omega;\mathbb{R}^d)\) for \(0 < s \leq 1\). Assume that $K * f \in L^{\infty}\left(0, T; W^{2,\infty}(\Omega; \mathbb{R}^d)\right)$ and
\[\int_{0}^{T} \left( \int_{\mathcal{D}} ( \left|\nabla_{v} \log f \right|^{2} + \left| \frac{1}{f} \nabla^{2}_{zv} f\right|^{2} ) f \right)^{\frac{1}{2}} < \infty,
\qquad
\sup_{x_\star}\int \bigl|\nabla_{v_\star}\log f(z_\star)\bigr|^2 f(z_\star)\,dv_\star
< \infty.
\]
In addition, assume that there exists $0<\rho_0<1$ such that
\[
\sum_{n=1}^\infty \rho_0^n\,\|\bar C_{N,n}(T)\|_{L^2_f}\leq N^{-1/2}.
\]
Then there exists $C>0$ and $\Delta t_1 >0$ depending on {$\rho_0$}, and various other constants but not $N$ such that the components $h^{(1)}_{N,n}$ and $w^{(1),i_1}_{N,n}$ satisfy the following bounds for all $t\in [T-\Delta t_1,\ T]$
\begin{equation}\label{bound:a_Nn}
\int_t^T\| h^{(1)}_{N,n}(\tau)\|_{L^2_f} d\tau \lesssim C^{n}\, N^{-\frac{s}{2}},
\end{equation}
and
\begin{equation}\label{bound:b_Nn}
\int_t^T\| w^{(1),i_1}_{N,n}(\tau)\|_{L^2_f} d\tau \lesssim C^n\, N^{-\frac{1}{2}}.
\end{equation}
\end{proposition}

\begin{proof}
From the assumptions on $K$ and $f$, we derive the following estimates on the weighted derivatives involving $f$ and the mollified kernels
\begin{align*}
    \|\nabla (f\, V_f^\delta)\|_{L^2_f} &\lesssim \delta^{\,s - 1},\\[4pt]
    \|W_f^\delta\|_{L^2_f} &\lesssim \delta^{\,s}.
\end{align*}
By Lemma \ref{lem:tensor_all_derivatives}, the semigroups $G^{(n)}$ , $G^{(n),0,i_1}$ and $G^{(n),i_1,i_1}$ are bounded on $L^2$. So applying the $L^2_f$ norm on the equations for $w_{N,n}^{(1),i_1}$ and $h_{N,n}^{(1)}$ from Lemma \ref{lem:decomposition-h1-w1} yields the following inequalities 
\begin{equation}\label{ineq:w_Nn}
\begin{aligned}
\|w^{(1),i_1}_{N,n}(t)\|_{L^2_f}
\;\lesssim\;&
C_1 \Bigg( \Lambda\!\left(\frac{n+1}{N}\right)^{1/2} \\
\; &+n\!\int_t^T \|w^{(1),i_1}_{N,n}(\tau)\|_{L^2_f}\, d\tau
\;+\;
n\!\int_t^T \|w^{(1),i_1}_{N,n+2}(\tau)\|_{L^2_f}\, d\tau \Bigg),
\end{aligned}
\end{equation}
since the $w_{N,n}$ all vanish at time $T$. Correspondingly, we have for the $h_{N,n}^{(1)}$
\begin{equation}\label{ineq:h_Nn}
\begin{aligned}
&\|h^{(1)}_{N,n}(t)\|_{L^2_f}
\;\lesssim\; \|\bar C_{N,n}(T)\|_{L^2_f}+C_1 \Bigg(
n\,\Lambda\!\left(\frac{n+1}{N}\right)^{1/2}
\\ & +\;
n\!\int_t^T \|h^{(1)}_{N,n}(\tau)\|_{L^2_f}\, d\tau +n\delta^s(T-t)
\;+\;
n\!\int_t^T \|w^{(1),j}_{N,n}(\tau)\|_{L^2_f}\, d\tau \\
&+\;
n\!\int_t^T \|h^{(1)}_{N,n+2}(\tau)\|_{L^2_f}\, d\tau
\;+\;
n\,\delta^{\,s-1}\!\int_t^T \|w^{(1),j}_{N,n+2}(\tau)\|_{L^2_f}\, d\tau \Bigg).
\end{aligned}
\end{equation}
We also note that the term $\Lambda \!\left(\frac{n+1}{N}\right)^{1/2}$ in the inequalities above arises from bounding each of the remainder terms on pages 12-13, using the assumptions on $K$ and $f$.
The rest of the proof is devoted to solving the two hierarchies of estimates above to derive the desired bounds.\\

\noindent
We start with hierarchy~\eqref{ineq:w_Nn}, and introduce the following time-integrated quantities
\[
b_{N,n}(t)
:= \sup_{i_1} \int_{t}^{T} \|w^{(1),i_1}_{N,n}(\tau)\|_{L^2_f}\,d\tau.
\]
Inequality~\eqref{ineq:w_Nn} then becomes
\begin{equation}\label{ineq:b_Nn}
\partial_t b_{N,n}(t)
\;\gtrsim\;
- C \Lambda (n+1)\left( b_{N,n}(t)+ b_{N,n+2}(t) + N^{-\frac{1}{2}} \right) .
\end{equation}
We can solve hierarchy~\eqref{ineq:b_Nn} using the generating function 
\begin{equation}\label{def:G}
Z_N(t,\rho)
:= \sum_{n=1}^{\infty} \rho^n\left(\, b_{N,n}(t)+ C(n+1) N^{-\frac{1}{2}}\int_0^t \Lambda\right),
\end{equation}
which we only consider for $0<\rho<1$ so that $\sum_n \rho^n\,n$ is finite. Using the bound~\eqref{ineq:b_Nn}, we obtain
\[
\begin{aligned}
\partial_t Z_N(t, \rho)
&\geq- C \Lambda\, \sum_{n=1}^{\infty} \rho^n(n+1) \left( b_{N,n}(t) + b_{N,n+2}(t) \right). \\
\end{aligned}
\] 
Now we rewrite the second sum in the following way
\begin{align*}
\partial_t Z_N(t,\rho)
&\geq -C \Lambda\, \rho 
\sum_{n=1}^{\infty} (n+1) \rho^{n-1}
 b_{N,n}  \\
&\quad - C \Lambda\, \frac{1}{\rho} 
\sum_{n=1}^{\infty} (n+1) \rho^{n+1}\,b_{N,n+2} \\
& \quad \geq -C \Lambda\,\bigg( ( \rho + \frac{1}{\rho} ) \partial_{\rho} Z_N -  ( 1 - \frac{1}{\rho^2} ) Z_N \bigg),
\end{align*} 
which yields
\[
\partial_t Z_N(t, \rho)
\geq -2C\, \frac{\Lambda}{\rho}\, \partial_{\rho} Z_N.
\]
We can solve the differential inequality above using the method of characteristics, define
\[
\frac{d}{dt}\bar\rho(t,s,\rho)=-2C\, \frac{\Lambda}{\bar\rho},\quad\bar\rho(s,s,\rho)=\rho.
\]
This is easy to solve and gives
\[
\bar\rho^{2}=\rho^2-4C\,(t-s)\, \Lambda.
\]
This implies that for $t_1\geq t_2$
\[
Z_{N}(s, \bar \rho(t_1,s,\rho)) \leq Z_{N}(t, \bar\rho(t_2,s,\rho)),
\] 
and consequently
\[
Z_{N}(t, \sqrt{\rho^2 - 4C\,(T-t)\, \Lambda}) \leq Z_{N}(T, \rho).
\]
Choose \( \Delta t_1^{(1)} < \frac{\rho^2}{8\,C\,\Lambda} \) and assume \( T - t \leq \Delta t_1^{(1)} \) so that, 
\[
 \sqrt{\rho^2 - 4C\,(T-t)\, \Lambda} \geq \sqrt{\rho^{2} - 4C \Lambda\Delta t_1^{(1)}}\geq \frac{\rho}{\sqrt{2}}.
 \]
We also have the following bound at $t=T$
\begin{align*}
Z_{N}(T, \rho) 
&\lesssim \sum_{n=1}^\infty \rho^n \left(  b_{N,n}(T) 
+ C(n+1) N^{-\frac{1}{2}}  \right) \\
&\lesssim N^{-\frac{1}{2}}\sum_{n=1}^\infty \rho^n C(n+1)  \\
&\lesssim  N^{-\frac{1}{2}} (1 - \rho)^{-2}\lesssim N^{-\frac{1}{2}},
\end{align*}
for a fixed choice of $\rho$.
\noindent
Putting everything together, we get for all \(t \in [T-\Delta t_1^{(1)},T]\):
\[
\sum_{n=1}^\infty \left( \frac{\rho}{\sqrt{2}} \right)^n\,b_{N,n}(t)  
\leq Z_{N}(t,\sqrt{\rho^2 - 4C\,(T-t)\, \Lambda} ) 
\leq Z_{N}(T, \rho) 
\lesssim N^{-\frac{1}{2}},
\]
which directly yields, by taking any value for $\rho<1$, that
\begin{equation}\label{bound:b_Nn_pointwise}
 b_{N,n}(t) \lesssim \left(\frac{\sqrt{2}}{\rho}\right)^n N^{-\frac{1}{2}}\lesssim C^n\,N^{-1/2}.
\end{equation} 
Moving on to hierarchy~\eqref{ineq:h_Nn}, we first use the bound~\eqref{bound:b_Nn_pointwise} and pick \(\delta := N^{-\frac{1}{2}}\). Following similar steps, we obtain the following differential hierarchy
\begin{equation}\label{ineq2:h_Nn}
\begin{aligned}
&\partial_t a_{N,n}(t)
\;\gtrsim\;
-\|\bar{C}_{N,n}(T)\|_{L^2_f}-n\Lambda\!\left(\frac{n+1}{N}\right)^{\!1/2}-n\delta^s(T-t)-n\delta^{s-1}\,C^n\, N^{-\frac{1}{2}}\\
&\qquad\qquad- C(n+1)\left( a_{N,n}(t) +\,a_{N,n+2}(t)\right)\\
&\quad\;\gtrsim - \| \bar{C}_{N,n}(T)\|_{L^2_f} -C(n+1) \left( a_{N,n}(t) + a_{N,n+2}(t) + N^{-\frac{s}{2}}\, C^n\right),
\end{aligned}
\end{equation}
where we denote
\[
a_{N,n}(t)
:= \int_{t}^{T} \|h_{N,n}^{(1)}(\tau)\|_{L^2_f}\,d \tau.
\]
We solve hierarchy~\eqref{ineq2:h_Nn} using again the corresponding generating function
\[
\widetilde{Z}_N(t,\rho)
:= \sum_{n=1}^{\infty} \rho^n
\Big(
a_{N,n}(t)
+ (n+1)\,C^n\,
N^{-\frac{s}{2}}\!\int_0^t \Lambda
+ \|\bar{C}_{N,n}(T)\|_{L^2_f}
\Big).
\]
Similarly by applying the bound~\eqref{ineq:h_Nn} and differentiating w.r.t \(t\) we obtain the following inequality 
\[
\partial_t \widetilde{Z}_N(t,\rho) \geq-2C \Lambda(t) \frac{1}{\rho} \partial_{\rho} \widetilde{Z}_N.
\]
We can solve this differential inequality again through characteristics
\[
\widetilde{Z}_N(t, \sqrt{\rho^2 - 4C\,(T-t)\, \Lambda}) \leq \widetilde{Z}_N(T, \rho).
\]
We naturally choose $\rho<1/C$ and $\rho<\rho_0$ such that
\[
\sum_{n=1}^\infty \rho^n\,\|\bar C_{N,n}(T)\|_{L^2_f}\leq N^{-1/2},
\]
which implies that 
\begin{align*}
\widetilde{Z}_{N}(T, \rho) 
&\lesssim \sum_{n=1}^\infty \rho^n \left(  a_{N,n}(T) 
+ C(n+1)\,C^n\, N^{-\frac{s}{2}} +\|\bar C_{N,n}(T)\|_{L^2_f}  \right) \\
&\lesssim N^{-1/2} +  N^{-\frac{s}{2}}\sum_{n=1}^\infty \rho^n\, C(n+1)\,C^n \lesssim N^{-\frac{s}{2}}.
\end{align*}
Let \( \Delta t_1 =\min\left(\Delta t_1^{(1)},\;\frac{\rho^2}{8\,C\,\Lambda}\right)\). We then obtain the following bound for $n \geq 1$, and for all \(t \in [T-\Delta t_1,T]\)
\[
a_{N,n}(t)\lesssim C^n N^{-s/2}.
\]
This concludes the proof.
\end{proof}

\subsection{Higher Order Decomposition of the Leading Terms and the Remainder}
To prove Theorem~\ref{thm2}, we first establish the following lemmas, which yield an explicit decomposition of both the leading--order and remainder terms in the hierarchy \eqref{eq:hierarchy_equation}, according to the regularity level of the kernel.
\begin{lemma}\label{eq:lemma1}
Let \( k \geq 1 \), and suppose \( k - 1 < s \leq k \). Let \( G_K \in W^{k-1,\infty}(\Omega ; \mathbb{R}^d) \cap H^{s}(\Omega; \mathbb{R}^d) \). 
Assume that the dual cumulants are decomposed up to their \(k-1\) derivatives, as follows
\[
\bar{C}_{N,n} := h^{(k-1)}_{N,n}+\sum_{q=1}^{k-1} \sum_{i_1,\ldots,i_q=1}^n div_{z_{i_1}\ldots z_{i_q}} w^{(k-1),i_1,\ldots,i_q}_{N,n}.
\]
Then one can further derive for any particle index $i\neq 2$
\begin{align*}
& G_K(x_i,x_2)\,\bar{C}_{N,n}=G_K(x_i,x_2)\,h_{N,n}^{(k-1)}\\
&\quad +
\sum_{q=1}^{k-1} \sum_{m=0}^{q}\sum_{\ell=0}^{q-m} \alpha_{m,\ell,q} \mathbb{I}_{\substack{i_{m+1},\ldots, i_{m+\ell}=i \\ i_{m+\ell+1},\ldots,i_q=2}}
\sum_{i_1,\ldots,i_{m}=1}^{n} \nabla_{z_{i_1}\dots z_{i_{m}}}
\Big[
(\nabla^{q-m} G_K)\,
w_{N,n}^{(k-1),i_1,\ldots,i_{q}}
\Big] ,
\end{align*}
where $\alpha_{q,\ell,m}$ is given by
\[
\alpha_{m,\ell,q} := (-1)^{q-m} \binom{q}{m}\,\binom{q-m}{\ell}.
\]
We also define here for $m=0$, $ \sum_{i_1,\ldots,i_{0}=1}^{n} 
\nabla_{z_{i_1}\dots z_{i_{0}}}:=1$. 
 \end{lemma}

\begin{proof}
The proof is a straightforward repeated application of the product rule.
More precisely, fix a term with \(q\) derivatives in the decomposition of \(\bar C_{N,n}\), where \(1\le q\le k-1\). Each term on the right--hand side of the desired equation corresponds to the case where exactly \(0\le m\le q\) derivatives act outside the brackets, while the remaining \(q-m\) derivatives fall on \(G_K\). The identity naturally holds up to permutations of the indices of
\(w_{N,n}^{(k-1),i_1,\cdots,i_q}\), so we can freely assume by symmetry that the first \(m\) indices correspond to the derivatives outside the brackets. Therefore, for \(0 \leq m \leq q\), the terms with \(m\) derivatives \(z_{i_{1}}\cdots z_{i_{m}}\) outside will be of the form
\[
\nabla_{z_{i_1}\dots z_{i_{m}}}\bigl(\nabla_{x_{i_{m+1}}\cdots x_{i_q}}G_K\bigr)
w_{N,n}^{(k-1),i_1,\dots,i_q},
\]
where the interior derivatives \((i_{m+1},\ldots,i_q)\) are confined to either \(i\) or \(2\). Therefore we can choose \(\ell\) times \(i\) where $0 \leq \ell \leq q-m$. This gives us $\binom{q-m}{\ell}$ options for $i$'s, then each option fixes the 2's. Thanks to symmetry, we may again assume that \(i_{m+1},\ldots,i_{m+\ell}=i\), while \(i_{m+\ell+1},\ldots,i_q=2\).

\noindent
The coefficients \(\alpha_{m,\ell,q}\) account for the multiplicity of identical terms generated by repeated applications of the product rule and our use of symmetry. To calculate the number of terms with \(m\) derivatives outside, we first count how many possible placements for the free indices \(i_j\)'s we have. The answer is \({q \choose m}\). For each such placement, we then count the number of ways to choose which \(\ell\) of the remaining \(q-m\) indices are equal to \(i\). There are \({q-m \choose \ell}\) such terms, which yields the formula for \(\alpha_{m,\ell,q}\).
\end{proof}
We note that the statement of the lemma can be simplified by introducing the notation $w^{(k-1),i_1,\ldots,i_0}_{N,n}=h_{N,n}^{(k-1)}$ when $q=0$. In that case one may just write     
\begin{align*}
& G_K(x_i,x_2)\,\bar{C}_{N,n}\\
&\quad =
\sum_{q=0}^{k-1} \sum_{m=0}^{q}\sum_{\ell=0}^{q-m} \alpha_{m,\ell,q} \mathbb{I}_{\substack{i_{m+1},\ldots, i_{m+\ell}=i \\ i_{m+\ell+1},\ldots,i_q=2}}
\sum_{i_1,\ldots,i_{m}=1}^{n} \nabla_{z_{i_1}\dots z_{i_{m}}}
\Big[
(\nabla^{q-m} G_K)\,
w_{N,n}^{(k-1),i_1,\ldots,i_{q}}
\Big].
\end{align*}
For convenience in the proof of the next lemma, we interchange the roles of the indices $q$ and $m$. Reindexing the summation yields the equivalent representation
\begin{equation}\label{eq:G_k_on_cumulants}
\begin{aligned}
G_K(x_i,x_2)\,\bar{C}_{N,n}
&=
\sum_{m=0}^{k-1}\sum_{q=0}^{m}\sum_{\ell=0}^{m-q}
\alpha_{q,\ell,m}\,
\mathbb{I}_{\substack{
i_{q+1},\ldots,i_{q+\ell}=i \\
i_{q+\ell+1},\ldots,i_m=2
}} \\
&\qquad 
\sum_{i_1,\ldots,i_q=1}^{n}
\nabla_{z_{i_1}\dots z_{i_q}}
\Big[
(\nabla^{m-q}G_K)\,
w_{N,n}^{(k-1),i_1,\ldots,i_m}
\Big].
\end{aligned}
\end{equation}
We make use of this notation below as it simplifies the formulas.\\ 

We can now establish iterative equations on the remainder \(R_{N,n}\) terms using Lemma \ref{eq:lemma1}. The core strategy is to control the components of
\(R_{N,n}\) by expressing them in terms of the previous dual components \(\{w_{N,n}^{(k-1),i_1,...,i_{m-1}}\}_{m=1}^k\); here we denote \(w_{N,n}^{(k-1),i_0}:= h_{N,n}^{(k-1)}\). This iterative approach allows us to leverage existing 
bounds that were obtained in a previous iteration.
\begin{lemma}\label{lemma3}
Under the assumptions of Theorem~\ref{thm2}, there exists a decomposition of \(R_{N,n}\)
\begin{equation}\label{eq:Remainder_decomposition}
R_{N,n}
= R_{N,n}^{(k)}
+ \sum_{q=1}^k \;\sum_{i_1,\dots,i_q=1}^n
\mathrm{div}_{z_{i_1}}\cdots \mathrm{div}_{z_{i_q}}
\, R_{N,n}^{(k),i_1,\dots,i_q},
\end{equation}
where up to permutation of the dual indices,
the components \(R_{N,n}^{(k)}\) and \(R_{N,n}^{(k),i_1,\dots,i_q}\) satisfy the following
representations for some families of coefficients
$A_{\alpha,m,l}^{(q,1)}$, $A_{\alpha, m,l}^{(q,2)}$ with $\alpha=-1,\,0,\,+1$. For $q=0$
\[
R_{N,n}^{(k)}
= \Lambda\!\left(\tfrac{1}{nN}\right)^{\frac12}
\sum_{m=0}^{k-1} \;\sum_{\ell=0}^{m}\sum_{i\neq j} \mathbb{I}_{\substack{i_{1},\ldots, i_{\ell}=i \\ i_{\ell+1},\ldots,i_m=j}}\sum_{\alpha=-1,\,0,\,1}
A_{\alpha,m,\ell}^{(0,1)}\,
w_{N,n+\alpha}^{(k-1),i_1,\dots,i_{m}},
\]
and for every \(1\le q\le k\),
\[
\begin{aligned}
R_{N,n}^{(k),i_1,\dots,i_q}
= \Lambda\,\left(\tfrac{1}{nN}\right)^{\frac12}\,\Bigg(
&\sum_{m=q}^{k-1} \sum_{\ell=0}^{m-q}\sum_{i\neq j} \mathbb{I}_{\substack{i_{q+1},\ldots, i_{q+\ell}=i \\ i_{q+\ell+1},\ldots,i_m=j}}\,\sum_{\alpha=-1,\,0,\,1} 
A_{\alpha,m,\ell}^{(q,1)}\,
w_{N,n+\alpha}^{(k-1),i_1,\dots,i_{m}}\\
&+\sum_{m=q-1}^{k-1} \sum_{\ell=0}^{m-q+1}\sum_{i\neq j}
\mathbb{I}_{\substack{
i_{q},\ldots, i_{q+\ell-1}=i \\
i_{q+\ell},\ldots,i_m=j
}}
\sum_{\alpha=-1,\,0,\,1}
A_{\alpha,m,\ell}^{(q,2)}\,
w_{N,n+\alpha}^{(k-1),i_1,\dots,i_m}
\Bigg),
\end{aligned}
\]
where the first term on the right-hand side is defined to be \(0\) for \(q=k\).
\end{lemma}

\begin{proof}
The proof is obtained by applying equation \eqref{eq:G_k_on_cumulants} from Lemma~\ref{eq:lemma1} to each contribution in the
definition of \(R_{N,n}\), choosing \(G_K\in\{V_f,K,K\ast f\}\) depending on the term and recalling that $R_{N,n}$ depends on $\bar C_{N,n-1},\;\bar C_{N,n}$ and $\bar C_{N,n+1}$. This yields a representation of the form
\[
\begin{aligned}
&R_{N,n}
= \Lambda\,\left(\tfrac{1}{nN}\right)^{\frac12}\,\Bigg(
\sum_{m=0}^{k-1} \;\sum_{\ell=0}^{m} \alpha_{0,\ell,m}
\sum_{i\neq j}
\mathbb{I}_{\substack{i_{1},\ldots, i_{\ell}=i \\ i_{\ell+1},\ldots,i_m=j}}
\sum_{\alpha=-1,\,0,\,1}
A_{\alpha,m,\ell}^{(0,1)}\,
w_{N,n+\alpha}^{(k-1),i_1,\dots,i_{m}}
\\
& 
+\sum_{q=1}^k\sum_{i_1,\dots,i_q=1}^n
\mathrm{div}_{z_{i_1}}\cdots \mathrm{div}_{z_{i_q}}
\Bigg[
\sum_{m=q}^{k-1} \sum_{\ell=0}^{m-q}\alpha_{q,\ell,m}
\sum_{i\neq j}
\mathbb{I}_{\substack{i_{q+1},\ldots, i_{q+\ell}=i \\ i_{q+\ell+1},\ldots,i_m=j}}
\\
&\qquad
\sum_{\alpha=-1,\,0,\,1} 
A_{\alpha,m,\ell}^{(q,1)}\,
w_{N,n+\alpha}^{(k-1),i_1,\dots,i_{m}} +\sum_{m=q-1}^{k-1} \sum_{\ell=0}^{m-q+1}\alpha_{q,\ell,m}
\sum_{i\neq j}
\mathbb{I}_{\substack{
i_{q},\ldots, i_{q+\ell-1}=i \\
i_{q+\ell},\ldots,i_m=j
}}
\\
&\qquad
\sum_{\alpha=-1,\,0,\,1}
A_{\alpha,m,\ell}^{(q,2)}\,
w_{N,n+\alpha}^{(k-1),i_1,\dots,i_m}
\Bigg]
\Bigg),
\end{aligned}
\]
where each coefficient \(A_{\alpha,m,\ell}^{(q,1)},\
A_{\alpha,m,\ell}^{(q,2)},\) is a finite expression built from \(L_f^2\)-bounded derivatives
of the corresponding kernel \(G_K\). In this case, the notation \(A_{\alpha,m,\ell}^{(q,r)}\) is used only to
record the structure and size of these coefficients, rather than their explicit
form. The assumptions of Theorem \ref{thm2} ensure that these coefficients are bounded in \(L^2_f\). Indeed, the regularity condition \(G_K \in H^s(\Omega;\mathbb{R}^d)\cap \mathscr{R}_k(\Omega;\mathbb{R}^d)\) provides the required control on the derivatives of \(G_K\) entering the coefficients \(A_{\alpha,m,\ell}^{(q,r)}\), so that the resulting terms are controlled in the corresponding \(L_f^2\) norm. Note that some of the terms in $R_{N,n}$ involve an additional $\nabla_{v_i}$ derivative, which is why we make use of $A^{(q,1)}_{\alpha,m,\ell}$ for the terms without the additional derivative and of $A^{(q,2)}_{\alpha,m,\ell}$ for the terms with the additional derivative. Absorbing the $\alpha_{q,\ell,m}$'s into \(A_{\alpha,m,\ell}^{(q,1)},\
A_{\alpha,m,\ell}^{(q,2)}\), and identifying the various terms with the decomposition~\eqref{eq:Remainder_decomposition} and grouping all terms
with the same number of exterior divergences yields the stated expressions for
\(R_{N,n}^{(k)}\) and \(R_{N,n}^{(k),i_1,\dots,i_q}\).
\end{proof}
Following a similar argument as in Lemmas~\ref{eq:lemma1} and~\ref{lemma3}, we obtain the following result for the last two terms on the right--hand side of~\eqref{eq:Duhamel_rep}. We recall our definitions
\[
\mathcal J[\bar C_{N,n}]=\sum_{j=1}^n\,\int_{\mathcal{D}} V_f(z_\star,z_j)\,\bar C_{N,n}(z_{[n]\setminus \{j\}},z_\star)\,f(z_\star)\,dz_\star,
\]
and
\[
\begin{aligned}
\bar{\mathcal J}[\bar C_{N,n+2}]
&=\sqrt{(n+1)(n+2)}\\
&\qquad 
\int_{\mathcal D^2}
V_f(z_{n+1},z_{n+2}) \,
\bar C_{N,n+2}(z_{[n+2]}) 
f(z_{n+1})f(z_{n+2})\,
dz_{n+1}dz_{n+2}.
\end{aligned}
\]
One then has the following decompositions of the leading terms
\begin{lemma}\label{lemma2}
Let \(k\geq 1\) and consider a decomposition of the $\bar C_{N,n}$ given by
\begin{equation*}
\bar C_{N,n}
= h^{(k)}_{N,n}
+ \sum_{q=1}^k \sum_{i_1,\dots,i_q=1}^n
\mathrm{div}_{z_{i_1}}\cdots \mathrm{div}_{z_{i_q}}
\, w^{(k),i_1,\dots,i_q}_{N,n}.
\end{equation*} 
We then have the corresponding decomposition
\begin{equation}\label{J_decomposition}
\mathcal J[\bar C_{N,n}]
= \mathcal J^{(k)}[\bar C_{N,n}]
+ \sum_{q=1}^k \sum_{i_1,\dots,i_q=1}^n
\mathrm{div}_{z_{i_1}}\cdots \mathrm{div}_{z_{i_q}}
\,\mathcal J^{(k),i_1,\dots,i_q}[\bar C_{N,n}],
\end{equation}
where for \(q=0\)
\[
\mathcal J^{(k)}[\bar C_{N,n}]
= H^{0,k}+\sum_{j=1}^n \sum_{m=0}^{k-1 }\alpha_{0,m}\, \mathbb{I}_{\{i_{1},\ldots, i_{m}=n \}} 
\int
\left(\nabla^{m} \left(fV_f \right) \right)\,
w^{(k),i_1,\dots,i_m}_{N,n}(z_{[n]\setminus\{j\}},z_\star)
\,dz_\star,
\]
and for every \(1\le q\le k\),
\[
\begin{split}
&\mathcal J^{(k),i_1,\dots,i_q}[\bar C_{N,n}]
= -\mathbb{I}_{q=1}\,H^{1,k}\\
&\quad+\sum_{j=1}^n \sum_{m=q}^k \alpha_{q,m}
\mathbb{I}_{\{i_{q+1},\ldots, i_{m}=n \}} 
\int
\left(\nabla^{m-q} \left(fV_f \right) \right)\,
w^{(k),i_1,\dots,i_m}_{N,n}(z_{[n]\setminus\{j\}},z_\star)
\,dz_\star.
\end{split}
\]
The coefficients $\alpha_{q,m}=\alpha_{q,0,m}$ and $\alpha_{q,l,m}$ are defined in Lemma~\ref{eq:lemma1}. We also denote 
\[\begin{split}
H^{0,k}=&\alpha_{0,k}\,\sum_{j=1}^n \mathbb{I}_{i_1,\ldots,i_k=n}\,\int \nabla^{k-1}(K(x_\star-x_j)\,\nabla_{z_\star}^2 f(z_\star))\,w^{(k),i_1,...,i_k}_{N,n}(z_{[n]\setminus\{j\}}, z_\star)\,dz_\star\\
& -\alpha_{0,k}\,\sum_{j=1}^n \mathbb{I}_{i_1,\ldots,i_k=n}\,\int \nabla^k(K * f(x_\star)\,\nabla_{z_\star} f(z_\star))\,w^{(k),i_1,...,i_k}_{N,n}(z_{[n]\setminus\{j\}}, z_\star)\,dz_\star,
\end{split}\]
and
\[
\begin{aligned}
H^{1,k} =\;&
\alpha_{0,k}\,\sum_{j=1}^n \mathbb{I}_{i_1=j;\;\tilde i_1,\ldots,\tilde i_k=n}\,\int
\nabla^{k-1}\!\left(
K(x_\star-x_j)\,
\nabla_{z_\star} f(z_\star)
\right)
w^{(k),\tilde i_1,\ldots,\tilde i_k}_{N,n}
\big(z_{[n]\setminus\{j\}},z_\star\big)\,
dz_\star.
\\
\end{aligned}
\]
Similarly, one obtains a decomposition
\begin{equation}\label{J_decomposition2}
\bar{\mathcal J}[\bar C_{N,n+2}]
= \bar{\mathcal J}^{(k)}[\bar C_{N,n+2}]
+ \sum_{q=1}^k \sum_{i_1,\dots,i_q=1}^n
\mathrm{div}_{z_{i_1}}\cdots \mathrm{div}_{z_{i_q}}
\,\bar{\mathcal J}^{(k),i_1,\dots,i_q}[\bar C_{N,n+2}],
\end{equation}
where for \(q=0\), 
\[
\begin{aligned}
\bar{\mathcal J}^{(k)}[\bar C_{N,n+2}]
&= \sqrt{(n+1)(n+2)} \Bigg(
\sum_{m=1}^{k}\sum_{\ell=0}^{m}
\alpha_{0,\ell,m}\mathbb{I}_{\substack{
i_1,\ldots,i_\ell=n+1 \\
i_{\ell+1},\ldots,i_m=n+2
}}
\\
&\qquad 
\int
\bigl(\nabla^{m}(fV_f^\delta)\bigr)\,
w_{N,n+2}^{(k),i_1,\ldots,i_m}\,
dz_{n+1}dz_{n+2}
\\
&\qquad
+\int V_f^\delta h_{N,n+2}^{(k)}
f(z_{n+1})f(z_{n+2})\,
dz_{n+1}dz_{n+2}
\\
&\qquad
+\int W_f^\delta \bar C_{N,n+2}
f(z_{n+1})f(z_{n+2})\,
dz_{n+1}dz_{n+2}
\Bigg),
\end{aligned}
\]
and for all \(1 \leq q \leq k\) 
\[
\begin{aligned}
&\bar{\mathcal J}^{(k),i_1,\dots,i_q}[\bar C_{N,n+2}]
= \sqrt{(n+1)(n+2)} \,\sum_{m=q}^k\sum_{\ell=0}^{m-q} \alpha_{q,\ell,m}\,\mathbb{I}_{\substack{i_{q+1},\ldots, i_{q+\ell}=n+1 \\ i_{q+\ell+1},\ldots,i_m=n+2}}\\
&\qquad\qquad\qquad\int
\big(\nabla^{\,m-q}(fV_f^\delta)\big)(z_{n+1},z_{n+2})\,
w_{N,n+2}^{(k),i_1,\dots,i_m}(z_{[n+2]})\,
dz_{n+1}dz_{n+2}.
\end{aligned}
\]
\end{lemma}

\begin{proof} The proof follows the same combinatorial analysis as Lemma~\ref{lemma3}. We do have to decompose further however the term which would correspond to $m=k$ in the formula for $\mathcal{J}^{(k)}[\bar{C}_{N,n}]$ with $q=0$, just like we did in Lemma~\ref{lem:decomposition-h1-w1}. Namely, we focus on 
\[
\int \nabla^{k}(fV_f )(z_\star,z_j)\,
w^{(k),i_1,...,i_k}_{N,n}(z_{[n]\setminus\{j\}}, z_\star)\,dz_\star.
\]
The component of \(V_f\) containing \(K * f\) is already sufficiently regular and is kept in the derivative-free part. This produces
\[
I=-\sum_{j=1}^n \mathbb{I}_{i_1,\ldots,i_k=n}\,\int \nabla^k(K * f(x_\star)\,\nabla_{z_\star} f(z_\star))\,w^{(k),i_1,...,i_k}_{N,n}(z_{[n]\setminus\{j\}}, z_\star)\,dz_\star.
\]
We further write that
\begin{equation}
\nabla_{z_\star} (K(x_\star-x_j)\,\nabla_{z_\star} f(z_\star))=K(x_\star-x_j)\,\nabla_{z_\star}^2 f(z_\star)-\nabla_{z_j}(K(x_\star-x_j)\,\nabla_{z_\star} f(z_\star)).\label{spreadderivative}
\end{equation}
The first term yields
\[
II=\sum_{j=1}^n \mathbb{I}_{i_1,\ldots,i_k=n}\,\int \nabla^{k-1}(K(x_\star-x_j)\,\nabla_{z_\star}^2 f(z_\star))\,w^{(k),i_1,...,i_k}_{N,n}(z_{[n]\setminus\{j\}}, z_\star)\,dz_\star,
\]
which is combined with $I$ to form the term $H^{0,k}$.\\

\noindent
Finally, the second term on the right-hand side of Eq.~\eqref{spreadderivative} produces one derivative outside the integral, since the other functions in the integral do not depend on $x_j$. This hence adds one term to the case $q=1$, namely 
\[
\begin{aligned}
&
\sum_{j=1}^n
\nabla_{z_j}\,
\mathbb{I}_{\tilde i_1,\ldots,\tilde i_k=n}
\int
\nabla^{k-1}\!\left(
K(x_\star-x_j)\,
\nabla_{z_\star} f(z_\star)
\right)
w^{(k),\tilde i_1,\ldots,\tilde i_k}_{N,n}
\big(z_{[n]\setminus\{j\}},z_\star\big)\,
dz_\star,
\end{aligned}
\]
which leads to the term $H^{1,k}$.
\end{proof}
%
\subsection{Proof of Theorem~\ref{thm2} }
Having established the necessary setup above, we now proceed to the proof of Theorem~\ref{thm2}.
\begin{proof}
The proof is based on deriving the following bounds for the time-integrated components for \(n\ge k\) and \(1\le q\le k-1\) 
\begin{equation}\label{boundtimeintegrated}
\begin{split}
&\int_t^T \|w^{(k),i_1,...,i_k}_{N,n} (\tau)\|_{L^2_f}d\tau \lesssim {C}_k\, C^{nk}\, N^{-\frac{k}{2}},\\ 
&\int_t^T \|w^{(k),i_1,...,i_q}_{N,n}(\tau) \|_{L^2_f}d\tau \lesssim {C}_k\,C^{nk}\, N^{-\frac{k}{2}}, \\
&\int_t^T \|h^{(k)}_{N,n}(\tau) \|_{L^2_f}d\tau \lesssim {C}_k\,C^{nk}\, N^{-\frac{s}{2}},
\end{split}
\end{equation}
together with the following bounds for \(n<k\) 
\begin{equation}
 \int_t^T \left(\|w_{N,n}^{(k),i_1,...,i_k}(\tau) \|_{L^2_f}+  \|w_{N,n}^{(k),i_1,...,i_q}(\tau) \|_{L^2_f}+\|h_{N,n}^{(k)}(\tau) \|_{L^2_f}\right)\,d\tau\lesssim  C_k C^{nk}\,N^{-\max(k_0,n)/2}.\label{boundtimeintegratedk0}
\end{equation}
First, we proceed by induction by showing for all $\ell<k$ that
\begin{equation}\label{boundtimeintegratedl}
\begin{split}
&\int_t^T \left(\|w^{(\ell),i_1,...,i_\ell}_{N,n} (\tau)\|_{L^2_f}+ \|w^{(\ell),i_1,...,i_q}_{N,n}(\tau) \|_{L^2_f}+\|h^{(\ell)}_{N,n}(\tau) \|_{L^2_f}\right)\,d\tau \\& \lesssim {C}_\ell\, C^{n\ell}\, N^{-\frac{\min(\ell,\max(k_0,n))}{2}}.\\ 
\end{split}
\end{equation}
Then at the final level \(\ell=k\), we prove \eqref{boundtimeintegratedk0} for $n<k$, and \eqref{boundtimeintegrated} for $n \geq k$, with a slower decay in \(N\) when \(s<k\).

\medskip

\noindent\textbf{Base case (\(\ell=1\)).} The estimate for \(\ell=1\) has already been established 
in Proposition~\ref{component_bounds}, by decomposing $\bar{C}_{N,n}$ as in Lemma \ref{lem:decomposition-h1-w1}. The only difference is that, at the higher regularity level \(k>1\), there is no need to decompose the first derivative \(\nabla_{z_{i_1}}V_f\) as in \eqref{eq:V_f_decomposition}, since it is already bounded in $L^2_f$. Therefore the result \eqref{bound:a_Nn} holds with 1 instead of $s$. This proves \eqref{boundtimeintegratedl} with $\ell=1$.

\medskip

\noindent\textbf{Inductive step.}
For each \(2\leq \ell\leq k-1\), the argument proceeds as in the proof of Proposition~\ref{component_bounds}: We decompose the dual cumulants in the leading terms into their components up to order \(\ell\), while the dual cumulants in the remainder terms are decomposed only up to order \(\ell-1\). The bounds obtained at the preceding iteration for the \((\ell-1)\)-st components of the remainder then yield the corresponding new bounds for the \(\ell\)-th components of the leading terms. The final passage from \(k-1\) to \(k\) follows the same general strategy but requires additional care. We therefore present this last step in detail. assume that~\eqref{boundtimeintegratedl} holds for \(\ell=k-1\). We denote the time-integrated components for every $1 \le q \le k$
\[
W^{(k),q}_{N,n}(t)
:= \sup_{i_1,\ldots,i_q=1\dots n} \int_t^T \|w^{(k),i_1,\dots,i_q}_{N,n}(\tau)\|_{L^2_f}\,d \tau,
\qquad
H^{(k)}_{N,n}(t)
:= \int_t^T \|h^{(k)}_{N,n}(\tau)\|_{L^2_f}\,d\tau.
\]

\medskip

\noindent \underline{Step 1:} For $n \geq k$, we start by showing that we have the following estimates for $q\geq 1$ 
\begin{equation}
\label{hierarchy_bound_2}
\begin{aligned}
\frac{d}{dt} W^{(k),q}_{N,n}(t)
\geq\;&
-C_k\, n\,\sup_{m\geq q}\,W^{(k),m}_{N,n}-C_k\, n\,\sup_{m\geq q}\,W^{(k),m}_{N,n+2}-C_k\,C^{(k-1)\,n}\,n^{3/2}\, {N^{-k/2}},
\end{aligned}
\end{equation}
and for $q=0$
\begin{equation} \label{hierarchy_bound_3}
\begin{aligned}
\frac{d}{dt} H^{(k)}_{N,n}(t)\geq &-\|\bar{C}_{N,n}(T)\|_{L^2_f} -C_k\,n\,\sup_{0<m\leq k}\,W^{(k),m}_{N,n}-C_k\,n\,\delta^{s-k}\,\sum_{m\leq k} W^{(k),m}_{N,n+2}-n\,\,H^{(k)}_{N,n+2}\\
&-n\,\delta^s\,(T-t)-C_k\,C^{n\,(k-1)}\,{n^{3/2}}\,{N^{-k/2}}.
\end{aligned}
\end{equation}
We start by considering the remainder decomposition~\eqref{eq:Remainder_decomposition} from Lemma~\ref{lemma3} and the decompositions~\eqref{J_decomposition} and~\eqref{J_decomposition2} of \(\mathcal{J}[\bar C_{N,n}]\) and \(\bar{\mathcal{J}}[\bar C_{N,n+2}]\) from Lemma~\ref{lemma2}. Let us gather together the terms with the same number of derivatives by denoting 
\begin{equation} \label{g_n_decomp_k}
g_n^{(k)} :=
R^{(k)}_{N,n}
+ \mathcal{J}^{(k)}[\bar{C}_{N,n}]
+ \bar{\mathcal{J}}^{(k)}[\bar{C}_{N,n+2}] ,
\end{equation}
and for $1 \le q \le k$
\begin{equation}
g_n^{(k),i_1,\dots,i_q} \label{g_n_decomp}
:=
R^{(k),i_1,\dots,i_q}_{N,n}
+\mathcal{J}^{(k),i_1,\dots,i_q}[\bar{C}_{N,n}]
+\bar{\mathcal{J}}^{(k),i_1,\dots,i_q}[\bar{C}_{N,n+2}] .
\end{equation}
Then by Lemma~\ref{lem:tensor_all_derivatives}, the solution to the truncated hierarchy \eqref{eq:hierarchy_equation} has the specific form
\[
\begin{split}
    \bar{C}_{N,n}(t) = &G_{t,T}^{(n)} \bar{C}_{N,n}(T) - \int_t^T G_{t,\tau}^{(n)} g_n^{(k)}(\tau) d\tau\\
    &- \int_t^T \sum_{q=1}^k \sum_{i_1,...,i_q=1}^n \sum_{s \in S(i_1,...,i_q)} \nabla^s \left( G_{t,\tau}^{(n),s,i_{[q]}} g_n^{(k),i_1,...,i_q} (\tau)\right) d\tau,
  \end{split}  
  \]
for the $L^2$ bounded operators $G^{(n)}$, $\{ G^{(n),s,i_{[q]}}\}_{s,q}$ given by Lemma~\ref{lem:tensor_all_derivatives}. We then derive formulas for the $w^{(k),i_1,\ldots,i_q}_{N,n}$ by identifying on the left-hand and right-hand sides the terms with the same derivatives in front. For $q=k$, we just isolate the terms involving exactly $k$ derivatives, yielding
\[
w^{(k),i_1,...,i_k}_{N,n}(t) =-\int_t^T
\Big(
G_{t,\tau}^{(n),s,i_{[k]}}\,
g_n^{(k),i_1,\dots,i_k}(\tau) 
\Big)d\tau,\quad s=(i_1,\ldots,i_k).
\]
Similarly, for each $1 \le q \le k-1$, we obtain the contribution corresponding
to $r=q$, together with additional terms arising from the cases
$r=q+1,\dots,k$ in which exactly $|s|=q$ derivatives act. We first observe that we can identify
\[
\begin{aligned}
\sum_{i_1,\dots,i_q=1}^n 
\nabla_{z_{i_1}\cdots z_{i_q}}
w^{(k),i_1,\dots,i_q}_{N,n}(t)
&=-
\sum_{r=q}^k\sum_{i_1,\dots,i_r=1}^n
\int_t^T
\sum_{\substack{s\in S(i_1,\dots,i_r)\\ |s|=q}}
\nabla^s\!\Big(
G_{t,\tau}^{(n),s,i_{[r]}}\,
g_n^{(k),i_1,\dots,i_r}(\tau)
\Big)\, d\tau.
\end{aligned}
\]
This leads to the further choice
\[
\begin{aligned}
w^{(k),i_1,\dots,i_q}_{N,n}(t)
&=-
\sum_{r=q}^k
\sum_{i_{1},\dots,i_r=1}^n
\sum_{\substack{s\in S(i_1,\dots,i_r)\\ s=(i_1,\ldots,i_q)}}
\int_t^T
G_{t,\tau}^{(n),s,i_{[r]}}\,
g_n^{(k),i_1,\dots,i_r}(\tau)\, d\tau .
\end{aligned}
\]
Similarly, the component with no derivatives, i.e.\ $q=0$, yields
\[
\begin{aligned}
h^{(k)}_{N,n}(t)
&=
G_{t,T}^{(n)} \bar{C}_{N,n}(T)
- \int_t^T G_{t,\tau}^{(n)} g_n^{(k)}(\tau)\, d\tau
- \sum_{r=1}^k
\sum_{i_1,\dots,i_r=1}^n
\int_t^T
G_{t,\tau}^{(n),0,i_{[r]}}\,
g_n^{(k),i_1,\dots,i_r}(\tau)\, d\tau .
\end{aligned}
\]
Using identities~\eqref{g_n_decomp_k} and~\eqref{g_n_decomp} we obtain, up to permutations of the \(w\)'s 
\[
\begin{aligned}
w^{(k),i_1,\dots,i_k}_{N,n}(t)
&=-
\int_t^TG_{t,\tau}^{(n),s,i_{[k]}}\,
\Bigg(
R^{(k),i_1,...,i_k}_{N,n}
+
\mathcal{J}^{(k),i_1,...,i_k}[\bar{C}_{N,n}]
+
\bar{\mathcal{J}}^{(k),i_1,\dots,i_k}[\bar{C}_{N,n+2}]
\Bigg) \,d\tau,
\end{aligned}
\]
while for $1\le q \le k-1$ we obtain, 
\[
\begin{aligned}
&w^{(k),i_1,\dots,i_q}_{N,n}(t)= -\sum_{r=q}^{k}
\sum_{i_{1},\dots,i_r=1}^{n}
\sum_{\substack{s\in S(i_1,\dots,i_r)\\ s=(i_1,\ldots,i_q)}}
\\ & \qquad \qquad
\int_t^T
G_{t,\tau}^{(n),s,i_{[r]}}
\Bigg[R^{(k),i_1,...,i_r}_{N,n}+
\mathcal J^{(k),i_1,\dots,i_r}[\bar C_{N,n}]+
\bar{\mathcal J}^{(k),i_1,\dots,i_r}[\bar C_{N,n+2}]
\Bigg]
\,d\tau,
\end{aligned}
\]
and for $q=0$
\[
\begin{aligned}
h^{(k)}_{N,n}(t)
&=
G^{(n)}_{t,T}\,\bar{C}_{N,n}(T)
-
\int_t^T
G_{t,\tau}^{(n)}
\Bigg[
R^{(k)}_{N,n}+\mathcal{J}^{(k)}[\bar{C}_{N,n}] + \bar{\mathcal{J}}^{(k)}[\bar{C}_{N,n+2}]
\Bigg]\,
d\tau
\\[0.8em]
&\quad
-
\sum_{r=1}^k
\sum_{i_{1},\dots,i_r=1}^n
\int_t^T
G_{t,\tau}^{(n),0,i_{[r]}}
\Bigg[
R^{(k),i_1,...,i_r}_{N,n}
+ \mathcal J^{(k),i_1,\dots,i_r}[\bar C_{N,n}]+
\bar{\mathcal J}^{(k),i_1,\dots,i_r}[\bar C_{N,n+2}]
\Bigg]
\,d\tau. \\
\end{aligned}
\]
Now we simply substitute the terms $R_{N,n}^{(k),i_1,...,i_r}$, $\mathcal{J}^{(k),i_1,...,i_r}[\bar{C}_{N,n}]$, and $\bar{\mathcal{J}}^{(k),i_1,...,i_r}[\bar{C}_{N,n+2}]$, which are explicitly defined in Lemmas~\ref{lemma3} and~\ref{lemma2} into the equations above.
Absorbing all resulting constants into some constant $C_k$ depending only on $k$, we first obtain that
\[
\|H^{0,k}\|_{L^2_f}\leq C_k\,\sum_{j=1}^n \mathbb{I}_{i_1,\ldots, i_k=n}\|
w^{(k),i_1,\ldots,i_k}_{N,n}\|_{L^2_f},
\]
and
\[
\| H^{1,k}\|_{L^2_f} \leq C_k \sum_{j=1}^n \mathbb{I}_{i_1=j;\;\tilde i_1,\ldots,\tilde i_k=n}\|
w^{(k),\tilde i_1,\ldots,\tilde i_k}_{N,n}\|_{L^2_f}.\]
As a consequence,
\[
\|\mathcal{J}^{(k)}[\bar{C}_{N,n}]\|_{L^2_f}\leq C_k\,\sum_{j=1}^n \sum_{m=0}^k \mathbb{I}_{i_1,\ldots,i_m=n}\|
w^{(k),i_1,\ldots,i_m}_{N,n}\|_{L^2_f}
\]
since $\alpha_{0,m}=1$. Note that $\alpha_{q,0,m}=\binom{m}{q}\leq C_k$ for $m\leq k$, and hence
\[
\begin{split}
\|\mathcal{J}^{(k),i_1,...,i_q}[\bar{C}_{N,n}]\|_{L^2_f}&\leq C_k\,\sum_{j=1}^n \sum_{m=q}^k \mathbb{I}_{i_{q+1},\ldots,i_m=n}\|
w^{(k),i_1,\ldots,i_m}_{N,n}\|_{L^2_f}\\
&+C_k \sum_{j=1}^n \mathbb{I}_{q=1;\;i_1=j;\;\tilde i_1,\ldots,\tilde i_k=n}\|
w^{(k),\tilde i_1,\ldots,\tilde i_k}_{N,n}\|_{L^2_f}.
\end{split}
\]
Recall that $K\in H^s$ with $k-1<s\leq k$, along with the assumptions on $K$ and $f$, we may invoke the following estimates on the weighted
derivatives of $f$ and on the mollified interaction kernels:
\begin{align*}
\big\|\nabla^{k}(f\,V_f^\delta)\big\|_{L^2_f} &\lesssim \delta^{\,s-k},\\[4pt]
\big\|W_f^\delta\big\|_{L^2_f} &\lesssim \delta^{\,s}.
\end{align*}
This leads to
\[
\begin{aligned}
&\|\bar{\mathcal J}^{(k)}[\bar C_{N,n+2}]\|_{L^2_f} \leq C_k\,  n \sum_{m=0}^{k} \sum_{l=0}^{m} \delta^{s-k}\,\mathbb{I}_{\substack{i_{1},\ldots, i_{l}=n+1 \\ i_{l+1},\ldots,i_m=n+2}}\, 
\|w_{N,n+2}^{(k),i_1,\dots,i_m}\|_{L^2_f}+ n \,
\|h_{N,n+2}^{(k)}\|_{L^2_f} + n\,\delta^s,
\end{aligned}
\]
and for all \(1 \leq q \leq k\) 
\[
\begin{aligned}
&\|\bar{\mathcal J}^{(k),i_1,\dots,i_q}[\bar C_{N,n+2}]\|_{L^2_f}
\leq C_k\,n\,\sum_{m=q}^k\sum_{l=0}^{m-q} \,\mathbb{I}_{\substack{i_{q+1},\ldots, i_{q+l}=n+1 \\ i_{q+l+1},\ldots,i_m=n+2}}\,
\|w_{N,n+2}^{(k),i_1,\dots,i_m}\|_{L^2_f}.
\end{aligned}
\]
We perform similar steps on the remainder and bounding again all coefficients, including the $A_{\alpha,m,\ell}^{(r,1)}$ and $A_{\alpha,m,\ell}^{(r,2)}$, by some abstract constant~$C_k$, we find that 
\[
{\|}R_{N,n}^{(k)}{\|_{L^2_f}}
\leq  \tfrac{C_k}{(nN)^{\frac12}}\,
\sum_{m=0}^{k-1} \;\sum_{l=0}^{m}\sum_{i\neq j} \mathbb{I}_{\substack{i_{1},\ldots, i_{l}=i \\ i_{l+1},\ldots,i_m=j}}\,\sum_{\alpha=-1,\,0,\,1}
\|w_{N,n+\alpha}^{(k-1),i_1,\dots,i_{m}}\|_{L^2_f},
\]
and for every \(1\le q\le k\),
\[
\begin{aligned}
\|R_{N,n}^{(k),i_1,\dots,i_q}\|_{L^2_f}
\leq \tfrac{C_k}{(nN)^{\frac12}}\,\Bigg(
&\sum_{m=q}^{k-1} \sum_{l=0}^{m-q}\sum_{i\neq j} \mathbb{I}_{\substack{i_{q+1},\ldots, i_{q+l}=i \\ i_{q+l+1},\ldots,i_m=j}}\,\sum_{\alpha=-1,\,0,\,1} 
\|w_{N,n+\alpha}^{(k-1),i_1,\dots,i_{m}}\|_{L^2_f}\\
&+\sum_{m=q-1}^{k-1} \sum_{l=0}^{m-q+1}\sum_{i\neq j}
\mathbb{I}_{\substack{
i_q,\ldots,i_{q+l-1}=i\\
i_{q+l},\ldots,i_m=j
}}\,\sum_{\alpha=-1,\,0,\,1} 
\|w_{N,n+\alpha}^{(k-1),i_1,\dots,i_{m}}\|_{L^2_f}
\Bigg).
\end{aligned}
\]
Applying the $L^2_f$ norm to both sides of the above identities, we can estimate each term on the right-hand side
using the $L^2_f$--boundedness of the operators $G^{(n),s,i_{[r]}}$ and $G^{(n)}$. Summing over the remaining indices by taking the supremum as in the definition of the $W^{(k),q}_{N,n}$ , we obtain the following system of inequalities for $q\geq 1$
\[
\begin{aligned}
\frac{d}{dt} W^{(k),q}_{N,n}(t)
\geq\;&
-C_k\, n\,\sup_{m\geq q}\,W^{(k),m}_{N,n}-C_k\, n\,\sup_{m\geq q}\,W^{(k),m}_{N,n+2}-C_k\,\sup_{m\geq q-1} \sum_{\alpha=-1,\,0,\,1}\frac{n^{3/2}} {N^{1/2}}\,W^{(k-1),m}_{N,n+\alpha},
\end{aligned}
\]
and for $q=0$
\[
\begin{aligned}
\frac{d}{dt} H^{(k)}_{N,n}(t)\geq &-\|\bar{C}_{N,n}(T) \|_{L^2_f} -C_k\,n\,\sup_{m\leq k}\,W^{(k),m}_{N,n}-C_k\,n\,\delta^{s-k}\,\sum_{0<m\leq k} W^{(k),m}_{N,n+2}-C_kn\,\,H^{(k)}_{N,n+2}\\
&-n\,\delta^s\,(T-t)-C_k\frac{n^{3/2}}{N^{1/2}}\,\sup_{m\leq k-1} \,\sum_{\alpha=-1,\,0,\,1} W^{(k-1),m}_{N,n+\alpha}.
\end{aligned}
\]
To conclude this step, we only have to use the induction assumption on the $W^{(k-1),q}$ corresponding to~\eqref{boundtimeintegrated}. We need to be especially careful with the terms involving $W^{(k-1),m}_{N,n-1}$. Here we consider $n\geq k$, the bound~\eqref{boundtimeintegrated} shows that $W^{(k-1),m}_{N,k-1}\lesssim C_{k-1}\,C^{(k-1)^2}\,N^{-(k-1)/2}$ exactly which multiplied by $N^{-1/2}$ provides the correct bound.

\noindent
\underline{Step 2:} Prove the induction by solving the differential hierarchies in Step~1.
The solution is similar to that of Proposition~\ref{component_bounds}: the procedure is iterative, meaning we first solve~\eqref{hierarchy_bound_2} for $q=k$,
and then use the solution as input for the subsequent inequalities~\eqref{hierarchy_bound_2} for $1\leq q\leq k-1$ and~\eqref{hierarchy_bound_3}. We first prove the claim for $n \geq k$. We start by defining the generating series for \(k \geq 1\)
\[
\begin{aligned}
Z_{N}(t, \rho)
:=\;& 
\sum_{n=k}^{\infty} 
    \rho^{\,n}
    \Bigg(
        W^{(k),k}_{N,n}(t)
        \;+\;
        C_k\,n^{1/2}\,N^{-\frac{k}{2}}
        C^{n(k-1)}
    \Bigg),
\\[6pt]
&\text{with the limitation}\ 
\rho \lesssim \frac{1}{C^{k-1}},
\qquad 
\rho \in (0,1).
\end{aligned}
\]
Using the bound~\eqref{hierarchy_bound_2} for $q=k$, we obtain 
\[
\begin{aligned}
\partial_t Z_N(t, \rho)
&\geq- C_k\, \sum_{n=k}^{\infty} \rho^n\,(n+1)\, \left( W^{(k),k}_{N,n}(t) + W^{(k),k}_{N,n+2}(t) +C^{(k-1)\,n}\,n^{1/2}\, {N^{-k/2}} \right) \\
\end{aligned}
\] 
Now we rewrite the second sum in the following way
\begin{align*}
\partial_t Z_N(t,\rho)
&\geq -C_k\, \rho
\sum_{n=k}^{\infty} (n+1) \rho^{n-1} (
W^{(k),k}_{N,n}(t)+C_k\,n^{1/2}\,N^{-\frac{k}{2}}
        C^{n(k-1)}) \\
&\quad - C_k\, \frac{1}{\rho} 
\sum_{n=k}^{\infty} (n+1) \rho^{n+1}
W^{(k),k}_{N,n+2}(t) \\
& \quad \geq -C_k ( \rho + \frac{1}{\rho} ) \partial_{\rho} Z_N.
\end{align*} 
Since we only consider bounded $\rho$, this can be simplified into
\[
\partial_t Z_N(t, \rho)
\geq -C_k\, \frac{1}{\rho}\, \partial_{\rho} Z_N.
\]
As before, this differential inequality is easily solved using the method of characteristics. Define
\[
\frac{d}{dt}\bar\rho(t,s,\rho)=\frac{C_k}{\bar\rho},\quad\bar\rho(s,s,\rho)=\rho.
\]
This immediately gives
\[
\bar\rho^{2}=\rho^2+2\,C_k\,(t-s).
\]
On the other hand, we also have that for $t_1\geq t_2$
\[
Z_{N}(t_2, \bar \rho(t_2,s,\rho)) \leq Z_{N}(t_1, \bar\rho(t_1,s,\rho)),
\] 
and consequently
\[
Z_{N}(t, \sqrt{\rho^2 - 2C_k\,(T-t)}) \leq Z_{N}(T, \rho).
\]
Choose \( \Delta t_k^{(k)} =  \min\left(\Delta t_{k-1}, \frac{\rho^2}{4\,C_k}\right),\) and assume \( T - t \leq \Delta t_k^{(k)} \) so that, 
\[
 \sqrt{\rho^2 - 2C_k\,(T-t)} \geq \sqrt{\rho^{2} - 2C_k\, \Delta t_k^{(k)} }\geq \frac{\rho}{\sqrt{2}} .
 \]
We also have the following bound at $t=T$
\begin{align*}
Z_{N}(T, \rho) 
&\lesssim \sum_{n=k}^\infty \rho^n \left(  W^{(k),k}_{N,n} (T) 
+ C_k\,n^{1/2} C^{n\,(k-1)} N^{-\frac{k}{2}}   \right) \\
&\leq C_k\,N^{-\frac{k}{2}}\,\sum_{n=k}^\infty \rho^n\,C^{n(k-1)}\,(n+1)  \leq C_k\,N^{-\frac{k}{2}} ,
\end{align*}
for a fixed choice of $\rho<\left(\frac{1}{C}\right)^{k-1}$. Putting everything together, we get for all \(t \in [T-\Delta t_k^{(k)},T]\):
\[
\sum_{n=k}^\infty \left( \frac{\rho}{\sqrt{2}} \right)^n\,W^{(k),k}_{N,n} (t)\leq Z_{N}(T, \rho) 
\lesssim C_k\,N^{-\frac{k}{2}},
\]
which directly yields
\[
  W^{(k),k}_{N,n} (t)  \lesssim \left(\frac{\sqrt{2}}{\rho}\right)^n C_k\,N^{-\frac{k}{2}}\lesssim C_k\,C^{nk}\,N^{-k/2}.
\]
This yields the desired first bound for all \(n \geq k\) and all \(t \in [T-\Delta t_k^{(k)},\,T]\).\\

\noindent
For a general \(1 \le q \le k-1\), the estimate is obtained in the same manner, by substituting the already obtained estimates at higher levels $q+1,...,k$, so that~\eqref{hierarchy_bound_2} reduces to 
\[ 
\partial_tW^{(k),q}_{N,n}(t) \geq -C_k\,n\,\Bigg( C^{n\,(k-1)}\,N^{-\frac{k}{2}}+ W_{N,n}^{(k),q}(t) + W_{N,n+2}^{(k),q} (t) \Bigg) .
\]
This allows us to estimate $W^{(k),q}_{N,n}$ through a similar generating function. For every $q$, we obtain a possibly different time length $\Delta t_k^{(q)}$, and we set
\(\Delta t_k = \min_{0 \leq q \leq k} \left(\Delta t_{k}^{(q)}\right),\) which is the smallest time interval that is admissible for all \(0 \le q \le k\).\\

\noindent
At the final level $q=0$, we repeat the same argument but emphasize that~\eqref{hierarchy_bound_3} includes additional terms in $\delta$. Optimizing with $\delta := N^{-\frac{1}{2}}$ leads to the corresponding bound and concludes the proof.\\

\noindent
As for the case \(1\leq n<k\), set
\[
r_n:=\max\{k_0,n\}.
\]
The arguments of Steps 1 and 2 apply without modification, except that
the size of the remainder terms must be compared with the terminal-data
scale \(N^{-r_n/2}\). By the induction hypothesis at level \(k-1\),
for every \(\alpha\in\{-1,0,1\}\), the components entering the
remainder satisfy
\[
W_{N,n+\alpha}^{(k-1),m}
\lesssim
C_{k-1}C^{n(k-1)}
N^{-\frac{\min\{k-1,\max(k_0,n+\alpha)\}}{2}}.
\]
Since the remainder carries an additional factor \(N^{-1/2}\), its
contribution is bounded by
\[
N^{-\frac{
1+\min\{k-1,\max(k_0,n+\alpha)\}
}{2}}.
\]
The weakest bound corresponds to \(\alpha=-1\). Since \(n<k\), we have
\[
1+\min\{k-1,\max(k_0,n-1)\}
\geq \max\{k_0,n\}=r_n.
\]
Indeed, if \(n>k_0\), the left-hand side equals \(n\), whereas if
\(n\leq k_0\), it is at least \(k_0\). Thus every remainder term is
bounded by \(N^{-r_n/2}\), or better. Moreover, by assumption,
\[
\|\bar C_{N,n}(T)\|_{L_f^2}
\lesssim C^n N^{-r_n/2}.
\]
Applying the same generating-function argument as above therefore gives
\[
\int_t^T
\left(
\|w_{N,n}^{(k),i_1,\ldots,i_k}(\tau)\|_{L_f^2}
+
\|w_{N,n}^{(k),i_1,\ldots,i_q}(\tau)\|_{L_f^2}
+
\|h_{N,n}^{(k)}(\tau)\|_{L_f^2}
\right)\,d\tau
\lesssim
C_k C^{nk}N^{-r_n/2},
\]
which is precisely \eqref{boundtimeintegratedk0}.\\

\noindent

Finally, the pointwise estimates follow by substituting the time-integrated estimates obtained above into \eqref{hierarchy_bound_2} and \eqref{hierarchy_bound_3}. This yields the result.
\end{proof}

\section{Proof of Theorem~\ref{thm1}}
We first prove the second assertion, which gives the optimal \(N^{-1}\) rate. The first assertion follows by the same argument using the \(k=1\) estimates.
\begin{proof} 
We recall relation \eqref{prop_cumulants} between the second--order dual cumulant and the marginals. After rescaling the dual cumulants, we obtain
\[
\int_{\mathcal D^m}\psi^{\otimes m}
\left(F_{N,m}(T)-f^{\otimes m}(T)\right)=
-\frac{N}{\sqrt{\binom{N}{2}}}
\int_0^T
\left(
\int_{\mathcal D^2}
V_f\,\bar C_{N,2}\,f^{\otimes2}
\right)dt.
\]
For the terminal data of $h_N(T)$ given in \eqref{final_data_on_h_N}, the associated rescaled dual cumulants $\bar C_{N,n}$ satisfy the corresponding terminal condition
\begin{equation}\label{eq:terminal_rescaled_cumulants}
\bar C_{N,n}(z_1,\dots,z_n)\big|_{t=T}
=
\mathbf{1}_{n\le m}\,{N\choose n}^{-\frac12}{m\choose n}
\sum_{l=0}^{n}(-1)^{n+l}\sum_{\sigma\in\mathcal P_l^n}
\left(\int_{\mathcal D}\psi f\right)^{m-l}\psi^{\otimes l}(z_\sigma),
\end{equation}
which can be obtained by direct computation from the decomposition of $h_N(T)$. Notice that the final condition \eqref{eq:terminal_rescaled_cumulants} on the rescaled cumulants satisfies
\[\|\bar{C}_{N,n}(T) \|_{L^2_f} \leq C N^{-n/2} \qquad \text{for all } n \geq 0,\]
which satisfies the assumptions of Theorem \ref{thm2} with $k_0=0$ and $T\leq \Delta t_2$. By integrating by parts and using the estimates in Theorem~\ref{thm2} with \(s=2\), we obtain the optimal bound for all $ 0 \leq t \leq \Delta t_2$
\[
\left| \int_{\mathcal{D}^2} V_f\, \bar{C}_{N,2}\, f^{\otimes 2} \right|\lesssim  C N^{-1} .
\]
Putting both estimates together yields 
\[
\left| \int_{\mathcal{D}^m} \psi^{\otimes m} 
\left( F_{N,m}(T) - f^{\otimes m}(T) \right) \right|
\lesssim C N^{-1}.
\]
Although the estimate is first obtained for observables of the
form \(\psi^{\otimes m}\), polarization, linearity, and density of finite sums of product test functions in \(C_c(\mathcal{D}^m)\) yield the stated
\(C_c(\mathcal{D}^m)^\star\) bound. Hence, by the arbitrariness of \(T\) and \(m\),
\[
\left\| F_{N,m}(T)-f^{\otimes m}(T)\right\|_{C_c(\mathcal{D}^m)^\star}
\lesssim C N^{-1}.
\]
Similarly, under the weaker assumption on the kernel, the estimates in Proposition \ref{component_bounds} yield 
\[
\left\| F_{N,m}(T) - f^{\otimes m}(T) \right\|_{C_c(\mathcal{D}^m)^\star}
\lesssim C N^{-1/2}.
\]
\end{proof}

\section{Proof of Theorem \ref{direct_cumulants_bound}}\label{sec4}

Given the previously established estimates on the dual cumulants
\( \bar{C}_{N,n} \), we now aim to derive corresponding bounds for the direct cumulants. We begin by denoting $C_{N,n}^0 = C_{N,n}|_{t=0}$ and recalling the following identity, which relates the unrescaled direct cumulants
\( \kappa_{N,n} \) to the unrescaled dual cumulants \( C_{N,n} \)
\begin{equation}\label{eq:direct_backward}
\begin{split}
&\sum_{n=0}^{N} \binom{N}{n}
\int_{\mathcal{D}^n} \kappa_{N,n}(z_1, \ldots, z_n)(t)\,
C_{N,n}(z_1, \ldots, z_n)(t)\, f(t)^{\otimes n}
\\
&\quad=
\sum_{n=0}^{N} \binom{N}{n}
\int_{\mathcal{D}^n} \kappa_{N,n}^{0}(z_1, \ldots, z_n)\,
C_{N,n}^{0}(z_1, \ldots, z_n)\, (f^0)^{\otimes n}.
\end{split}
\end{equation}
This relation follows from the weak duality formulation of the solution,
together with the definitions of the direct and dual cumulants and their
associated cancellation properties. This relation directly implies the corresponding relation on the rescaled cumulants,
\begin{equation}\label{eq:direct_backward_rescaled}
\begin{split}
&\sum_{n=0}^{N} 
\int_{\mathcal{D}^n} \bar \kappa_{N,n}(z_1, \ldots, z_n)(t)\,
\bar C_{N,n}(z_1, \ldots, z_n)(t)\, f(t)^{\otimes n}
\\
&\quad=
\sum_{n=0}^{N}
\int_{\mathcal{D}^n} \bar \kappa_{N,n}^{0}(z_1, \ldots, z_n)\,
\bar C_{N,n}^{0}(z_1, \ldots, z_n)\, (f^0)^{\otimes n}.
\end{split}
\end{equation}

Our objective is to derive a propagation estimate on the direct cumulants
\( \kappa_{N,n} \) using suitable norms.

\begin{proof}
Let \(\{\kappa_{N,n}\}_{0\le n\le N}\) and \(\{ C_{N,n}\}_{0\le n\le N}\) denote the
direct and dual cumulants, related by the duality identity~\eqref{eq:direct_backward}. For \(1\leq i\leq n\), let \(\Pi_i\) denote the averaging operator in the
\(i\)-th variable with respect to \(f\), namely
\[
\Pi_i\phi(z_1,\ldots,z_n)
:=
\int_{\mathcal D}
\phi(z_1,\ldots,z_{i-1},z_i',z_{i+1},\ldots,z_n)
f(z_i')\,dz_i' .
\]
We define
\[
P_n:=(\mathrm{Id}-\Pi_1)\cdots(\mathrm{Id}-\Pi_n),
\]
and choose the terminal data at level \(n\) as
\[
\bar{C}_{N,n}(T)
:=
\mathbbm{1}_{n \geq k}N^{-n/2}P_n\operatorname{sgn}(\bar\kappa_{N,n}(T)).
\]
By construction, \(\bar{C}_{N,n}(T)\) satisfies the cancellation condition \eqref{eq:dual_cumulant_cancellation} in
each variable. Indeed, since \((\Pi_i)^2=\Pi_i\), we have
\[
\Pi_i(\mathrm{Id}-\Pi_i)=0,
\]
and therefore
\[
\Pi_i\bar{C}_{N,n}(T)=0,
\qquad 1\leq i\leq n.
\]
By the definition of \(\Pi_i\), this is precisely the cancellation condition. The final data above also satisfy the assumptions of Theorem \ref{thm2}: First, the reconstructed observable \(h_N(T)\) is bounded in \(L^\infty\). Indeed, since each \(\Pi_i\) is a contraction on \(L^\infty\), we have
\[
\|P_n\phi\|_{L^\infty}\leq 2^n\|\phi\|_{L^\infty}.
\]
Therefore,
\[
\|P_n\operatorname{sgn}(\bar\kappa_{N,n}(T))\|_{L^\infty}\leq 2^n.
\]
Using the relation
\[
C_{N,n}(T)=\binom{N}{n}^{-1/2}\bar C_{N,n}(T),
\]
we obtain
\[
\begin{aligned}
\|h_N(T)\|_{L^\infty}
&\leq
\sum_{n=k}^N
\sum_{\sigma\in P_n^N}
\|C_{N,n}(T,z_\sigma)\|_{L^\infty}  \\
&\leq
\sum_{n=k}^N
\binom{N}{n}
\binom{N}{n}^{-1/2}
N^{-n/2}
\|P_n\operatorname{sgn}(\bar\kappa_{N,n}(T))\|_{L^\infty}  \\
&\leq
\sum_{n=k}^N
2^n N^{-n/2}\binom{N}{n}^{1/2}
\leq
\sum_{n=k}^N
\frac{2^n}{(n!)^{1/2}}
\lesssim 1 .
\end{aligned}
\]
This is actually the most stringent requirement as it is the one that forces us to take the $N^{-n/2}$ factor in $\bar C_{N,n}(T)$, which in turn prevents us from being able to control any scale beyond $\bar \kappa_{N,n}=O(1)$. It is immediate to check that \eqref{finaldatadual} is satisfied with $k=k_0$ since $\bar C_{N,n}(T)=0$ for $n<k$ in that case. Then, for $T \leq \Delta t_k$, we have
\[
\|\bar C_{N,n}^0\|_{H^{-k}_{f}} \lesssim C_k\,C^{nk}\,N^{-\frac{k}{2}}.
\]
Here we use the negative Sobolev-type norm defined via the decomposition
\[
\|u\|_{H^{-k}_f(\mathcal{D}^{n})} := \inf \left\{ \|h^{(k)}\|_{L^2_f(\mathcal{D}^n)} 
+ \sum_{q=1}^{k} w^{(k),q} 
\;\middle|\; 
u = h^{(k)} + \sum_{q=1}^k \sum_{i_1,\cdots,i_q=1}^n \operatorname{div}_{z_{i_1}\ldots z_{i_q}} w^{(k),i_1,\ldots,i_q}
\right\},
\]
where we define \[ w^{(k),q} := \sup_{i_1,\ldots,i_q=1\ldots n}\|w^{(k),i_1,\ldots,i_q}\|_{L^2_f(\mathcal{D}^n)}. \]
From this and identity \eqref{eq:direct_backward_rescaled}, we obtain 
\begin{equation}\label{eq:direct_duality_bound}
\sum_{n=0}^{N}
\int_{\mathcal D^n} \bar{\kappa}_{N,n}(T)\, \bar{C}_{N,n}(T)\, f(T)^{\otimes n}
\;\lesssim\;
C_k N^{-k/2} \sum_{n=0}^{N} C^{nk}\Lambda_{N,n,k},
\end{equation}
with $\Lambda_{N,n,k}$ defined as in \eqref{Lambda_{N,n,k}}. 
Indeed, using the decomposition defining the \(H^{-k}_f\)-norm and integrating
by parts in the right-hand side of~\eqref{eq:direct_backward_rescaled}, the
derivatives are transferred onto \(f^0\bar\kappa_{N,n}^{0}\), producing precisely
the quantities entering \(\Lambda_{N,n,k}\).
Substituting the final data $\bar{C}_{N,n}(T)$ into the left--hand side, and using that \(P_n\) is self-adjoint in \(L^2_{f}\), while
\(\bar\kappa_{N,n}(T)\) satisfies the cancellation condition, we obtain \(P_n\bar\kappa_{N,n}(T)=\bar\kappa_{N,n}(T)\). Hence
\[
\begin{aligned}
\sum_{n=0}^{N}
\int_{\mathcal D^n}
\bar\kappa_{N,n}(T)\bar C_{N,n}(T) f(T)^{\otimes n}
&=
\sum_{n=k}^{N}
N^{-n/2}
\int_{\mathcal D^n}
P_n\bar\kappa_{N,n}(T)
\operatorname{sgn}(\bar\kappa_{N,n}(T))
f(T)^{\otimes n}  \\&=
\sum_{n=k}^{N}
N^{-n/2}
\|\bar\kappa_{N,n}(T)\|_{L^1_{f}},
\end{aligned}
\]
which gives
\[\sum_{n=k}^N N^{-n/2} \|\bar{\kappa}_{N,n}(T) \|_{L^1_f} \lesssim C_k N^{-k/2}\,
\sum_{n=0}^{N}
C^{nk}\Lambda_{N,n,k}.\]
Using the assumption $\sum_{n=0}^{N} C^{nk}\Lambda_{N,n,k} \lesssim 1$, we obtain
\begin{equation}\label{eq:direct_L1_bound}\| \bar{\kappa}_{N,k}(T) \|_{L^1_f}
\lesssim C_k\, .
\end{equation}
Since \(\kappa_{N,k}=\binom{N}{k}^{-1/2}\bar\kappa_{N,k}\) and \(k\) is fixed, we obtain the estimate 
\[
\|\kappa_{N,k}(T)\|_{L^1_f}
\lesssim C_k N^{-k/2}.
\]
By the arbitrariness of $T$, we obtain \eqref{eq:direct_L1_bound1}.
\end{proof}

\backmatter

\bmhead{Acknowledgements}

P.-E. J. and N. K. were partially supported by NSF DMS Grant 2508570.

\medskip

\noindent\textbf{Data Availability}\quad
Data sharing is not applicable to this article as no datasets were
generated or analysed during the current study.

\medskip

\section*{Declarations}

\noindent\textbf{Conflict of interest}\quad
The authors declare that they have no conflict of interest.

\begin{appendices}

\renewcommand{\theequation}{\thesection.\arabic{equation}}
\renewcommand{\theHequation}{\thesection.\arabic{equation}}
\setcounter{equation}{0}

\section{Weak-Duality Solutions}\label{secA1}
In this appendix, we recall the notion of weak-duality solutions for the
Liouville equation introduced in \cite{bresch2024duality}, which is
the appropriate solution framework for the duality method used in this
paper. Roughly speaking, a weak-duality solution is characterized through
its duality with a suitable bounded weak solution of the corresponding
backward Liouville equation for each prescribed terminal data. This
formulation is particularly useful for singular interaction kernels, for
which the forward Liouville equation may not be well defined in the usual
weak sense and uniqueness of solutions to the backward equation is not
available in general. We state below only the definitions and properties
needed in the present work, and refer the reader to
\cite{bresch2024duality} for a complete discussion of the construction,
existence, and further properties of weak-duality solutions.

\subsection*{The backward Liouville equation}
Given a final time $T>0$, we consider the backward Liouville equation
\begin{equation}\label{eq:backward_liouville}
\partial_t h_N
+ \sum_{i=1}^N \left(
v_i \cdot \nabla_{x_i} h_N
+ \frac{1}{N-1}\sum_{j\neq i} K(x_i-x_j)\cdot \nabla_{v_i} h_N
\right)
=0,
\end{equation}
with prescribed terminal condition $h_N(T)=h_N^T$. 

The equation above is the adjoint, in the sense of distributions, of the forward Liouville equation~\eqref{eq:Liouville_equation} satisfied by the $N$--particle density.

When the particle flow is sufficiently regular, the backward formulation is equivalent to the classical formulation of the Liouville equation. The advantage of the dual formulation is that it remains meaningful even for interaction kernels of low regularity, where the existence or uniqueness of the particle flow may not be available.

\subsection*{Weak duality solutions}
\begin{definition}
 We say that a family of probability densities \(F_N\) is a \textbf{weak duality solution} to the Liouville equation~\eqref{eq:Liouville_equation} with initial data \(F_N^0\) if, for every final time \(T>0\) and for every bounded terminal test function \(h_N^T\) vanishing at infinity, there exists a bounded weak solution \(h_N\) of the backward Liouville equation~\eqref{eq:backward_liouville} with terminal data \(h_N^T\) such that
\begin{equation}\label{eq:duality_identity}
\int_{\mathcal{D}^N} h_N^T(z)\, F_N^T(z)\,dz
=
\int_{\mathcal{D}^N} h_N^0(z)\, F_N^0(z)\,dz,
\end{equation}
where \(F_N^T:=F_N(T)\) and \(h_N^0:=h_N(0)\).
\end{definition} 

Existence of weak-duality solutions under the assumptions considered here is established in \cite{bresch2024duality}.

The identity above expresses the conservation of the pairing between the forward density $F_N$ and backward observable $h_N$ along the particle dynamics.

This notion of solution is the natural framework for the duality method developed in the present work. Indeed, by choosing appropriate families of backward solutions $h_N^T$, one can propagate quantitative estimates backward in time and thus obtain quantitative information on the forward \(N\)-particle density, including higher-order correlations and deviations from the mean-field limit.

\end{appendices}

\bibliography{sn-bibliography}







\end{document}